\documentclass[oneside, english, 11pt]{article}
\usepackage[utf8]{inputenc}
\usepackage[T1]{fontenc}
\usepackage{dsfont} 
\usepackage{amsmath, amsthm}
\usepackage{enumerate}
\usepackage{geometry}
\usepackage{mathtools}
\usepackage{longtable}
\usepackage{booktabs}
\usepackage{thmtools}
\usepackage{nicefrac}
\usepackage{hyperref}
\usepackage[title]{appendix}
\usepackage{paralist}
\newcommand{\footremember}[2]{%
   \footnote{#2}
    \newcounter{#1}
    \setcounter{#1}{\value{footnote}}%
}
\newcommand{\footrecall}[1]{%
    \footnotemark[\value{#1}]%
} \usepackage{setspace}
\usepackage[hang]{subfigure}
\usepackage{multirow}
\usepackage[format=hang]{caption} 

\hypersetup{
    colorlinks=true,
    linkcolor=black,
    citecolor=black,
    filecolor=black,
    urlcolor=black,
}
\usepackage[charter]{mathdesign}
\usepackage[mathcal]{eucal}
\usepackage{tabularx}
\newcolumntype{L}[1]{>{\raggedright\arraybackslash}p{#1}} 
\newcolumntype{C}[1]{>{\centering\arraybackslash}p{#1}} 
\newcolumntype{R}[1]{>{\raggedleft\arraybackslash}p{#1}} 

\newtheorem{theorem}{Theorem}
\numberwithin{theorem}{section}

\newtheorem{assumption}[theorem]{Assumption}
\newtheorem{corollary}[theorem]{Corollary}
\newtheorem{definition}[theorem]{Definition}
\newtheorem{example}[theorem]{Example}
\newtheorem{lemma}[theorem]{Lemma}

\newtheorem{proposition}[theorem]{Proposition}
\newtheorem{remark}[theorem]{Remark}

\numberwithin{equation}{section}
\newcommand{\R}{ {\mathbb{R}} }
\newcommand{\N}{ {\mathbb{N}} }
\newcommand{\Z}{ {\mathbb{Z}} }

\newcommand{\Q}{ {\mathbb{Q}} }
\newcommand{\E}{ {\mathbb{E}} }
\renewcommand{\P}{ {\mathbb{P}} }
\newcommand{\F}{ {\mathcal{F}} }
\newcommand{\h}[1]{\widehat{#1}}
\newcommand{\ind}{\mathds{1}}
\newcommand{\dd}{\,\mathrm{d}}
\newcommand{\ddi}{\mathrm{d}}

\newcommand{\ie}{^{-1}}
\newcommand{\e}{\mathrm{e}}
\newcommand{\de}{\overset{\mathrm{d}}{=}}
\DeclareMathOperator*{\argmin}{arg\,min}
\renewcommand{\L}{\widehat{\mathcal{L}}}
\newcommand{\cL}{\mathcal{L}}
\renewcommand{\epsilon}{\varepsilon}
\newcommand{\sep}{SEP$(X, \mu, \nu)$}
\newcommand{\s}{$\sigma$}
\newcommand{\sfrac}{\genfrac{}{}{}1}
\newcommand*{\medcup}{\mathbin{\scalebox{1.5}{\ensuremath{\cup}}}}%
\renewcommand{\tilde}{\widetilde}
\begin{document}

\title{A Free Boundary Characterisation of the Root Barrier for Markov Processes}
\author{Paul Gassiat * \footremember{alley}{*Universit\'e Paris Dauphine, PSL University, UMR 7534, CNRS, CEREMADE, 75016 Paris, France  (\href{gassiat@ceremade.dauphine.fr}{gassiat@ceremade.dauphine.fr})}  
\and Harald Oberhauser \footremember{trailer}{Mathematical Institute, University of Oxford, Woodstock Road, OX2 6GG, United Kingdom (\href{harald.oberhauser@maths.ox.ac.uk}{harald.oberhauser@maths.ox.ac.uk}, \href{christina.zou@maths.ox.ac.uk}{christina.zou@maths.ox.ac.uk})}  
\and Christina Z. Zou \footrecall{trailer} }
\maketitle
\begin{abstract}
We study the existence, optimality, and construction of non-randomised stopping times that solve the Skorokhod embedding problem (SEP) for Markov processes which satisfy a duality assumption.
These stopping times are hitting times of space-time subsets, so-called Root barriers.
Our main result is, besides the existence and optimality, a potential-theoretic characterisation of this Root barrier as a free boundary.  
If the generator of the Markov process is sufficiently regular, this reduces to an obstacle PDE that has the Root barrier as free boundary and thereby generalises previous results from one-dimensional diffusions to Markov processes.
However, our characterisation always applies and allows, at least in principle, to compute the Root barrier by dynamic programming, even when the well-posedness of the informally associated obstacle PDE is not clear. 
Finally, we demonstrate the flexibility of our method by replacing time by an additive functional in Root's construction. 
Already for multi-dimensional Brownian motion this leads to new class of constructive solutions of (SEP). 
\end{abstract}

\section{Introduction}
We study the Skorokhod embedding problem for Markov processes $X=(X_t)_{t\ge 0}$ evolving in a locally compact space $E$.
That is, given measures $\mu$ and $\nu$ on $E$, the task is to find a stopping time $T$ such that 
\begin{equation}\tag{\sep}
\mbox{ if } X_0 \sim \mu \mbox{ then } X_T \sim \nu.
\end{equation}
Throughout this article we are interested in non-randomised stopping times, that is $T$ is a stopping time in the filtration generated by $X$.
When $X$ is a one-dimensional Brownian motion, this problem has received much attention, partly due to its importance in mathematical finance \cite{hobson2011skorokhod}.
In this case, there exists a wealth of different stopping times that solve \sep, see \cite{obloj2004skorokhod} for an overview.
One of the most intuitive solutions is due to Root \cite{root1969existence}: for a one-dimensional Brownian motion and $\mu$,$\nu$ in convex order, there exists a space-time subset -- the so-called Root barrier -- such that its hitting time by $(t,X_t)$ solves \sep.
More recently, connections with obstacle PDEs \cite{cox2013root,dupire2005arbitrage,gassiat2015integral,gassiat2015root, ghoussoub2019pde}, optimal transport \cite{beiglbock2017optimal, beiglbock2017monotone,beiglbock2019fine, beiglbock2017complete, cox2015embedding, ghoussoub2019solution, guo2016monotonicity, guo2016optimal}, and optimal stopping \cite{cox2013root,de2018optimal} and extensions to the multi-marginal case \cite{beiglboeck2017geometry, cox2019root, richard2018root} have been developed.

However, already for multi-dimensional Brownian motion much less is known about solutions to \sep, see for example work of Falkner \cite{falkner1981distribution} that highlights some of the difficulties that arise in the multi-dimensional Brownian case.
For general Markov processes the literature gets even sparser: Rost \cite{rost1970,rost1971} developed a potential theoretic approach to previous work of Root, 
but in general this shows only the existence of a randomised stopping time for \sep\, when $\mu$ and $\nu$ are in balayage order.
Subsequent works of Chacon, Falkner, and Fitzsimmons, \cite{chacon1986barrier, falkner1981distribution,falkner1991stopping}, expand on these results and provide sufficient conditions for the existence of a non-randomised stopping time; however, in none of these works the question of how to compute these stopping times $T(\omega)$ for a given sample trajectory $X(\omega)$ is addressed.   
Another approach is the application of optimal transport to \sep\, as initiated by Beiglb\"ock, Cox, Huesmann \cite{beiglbock2017optimal}.
This covers Feller processes but verifying the assumptions can be non-trivial.
More importantly, the optimal transport approach currently only addresses the existence and optimality of a stopping time but not its computation. 
Besides these two approaches -- (Rost's) potential theoretic approach and the optimal transport approach -- we are not aware of a general methodology that produces solutions to \sep\,for Markov processes.  

\paragraph{Contribution.}

We focus on the large class of right-continuous transient standard Markov processes satisfying a duality assumption and absolute continuity of the semigroup. Our main result is Theorem~\ref{thm:main} which extends Rost's results and shows existence of a non-randomised Root stopping time and, more importantly, represents the Root barrier as a free boundary via the semigroup of the dual space-time process.
This allows to apply classical dynamic programming to calculate the Root barrier for a large class of Markov processes.
Theorem~\ref{thm:main} also implies that if a PDE theory is available that ensures the well-posedness of the free boundary problem formulated as PDE problem, then numerical methods for PDEs can be used to compute the barrier.
However, in general this requires much stronger assumptions on the Markov process, e.g.~when the generator involves non-local terms as is already the case for one-dimensional L\'evy processes, the well-posedness of such PDEs is an active research area.  

We present a series of examples of processes to which our result applies. The most important one is arguably multi-dimensional Brownian motion (or more generally, hypoelliptic diffusions), but we also discuss stable L\'evy processes and Markov chains on a discrete state space. In all these cases our result allows to compute the Root barrier, and we present several numerical experiments to illustrate this point.

Finally, we show that our approach is flexible enough to construct new classes of solutions to the Skorokhod embedding problem: instead of hitting times of the space-time process $(t,X_t)$, we discuss hitting times of $(A_t,X_t )$ where $A$ is an additive functional of $X$ of the form $\int_0^\cdot a(X_s) \dd s$.
We expect that such an approach holds in much greater generality for other functionals and leave this for further research.

\paragraph{Outline.}

The structure of the article is as follows: Section \ref{sec:notation} introduces notation and basic results from potential theory, Section~\ref{sec:characterization} contains the statement of our main result and Section \ref{sec:proof} contains its proof.  
Section~\ref{sec:examples} then applies this to concrete examples of Markov processes and computations of Root barriers. 
Section~\ref{sec:generalRoot} discusses how these results can be used to construct new solutions of \sep.
In Appendix \ref{app:basic} and Appendix \ref{app:red} we will present fundamental results from classical potential theory used throught this article and in Appendix \ref{app:hypo-elliptic} we discuss details around applying our result to Brownian motion in a Lie group.

\section{Notations and assumptions}\label{sec:notation}
We briefly recall notations from potential theory, mostly following the presentation in Blumenthal and Getoor \cite{blumenthal2007markov}. A detailed description can be found in Appendix \ref{app:basic}. 
Throughout, $E$ is a locally compact metric space with countable base and $\mathcal E$ is the Borel-\s-algebra on $E$. In addition, we write  $\mathcal{E}^\ast$ for the $\sigma$-algebra of universally measurable sets  and   $\mathcal{E}^n$ for the $\sigma$-algebra of nearly Borel sets, see Definition \ref{def:univmeas}.

Let $\left(\Omega, \F, (\F_t)_{t \geq 0}, (X_t)_{t\geq 0}, (\P^x)_{x \in E}\right)$ denote a filtered probability space that carries a stochastic process $X$.
To allow for killing we add an absorbing cemetery state $\Delta$ to the state space, that is we define $E_\Delta := E \cup \{\Delta\}$ and for all $t\geq 0$ if $X_t(\omega)=\Delta$, then $X_s(\omega)=\Delta$ for all $s>t$.
Denote with $\zeta := \inf\{t \geq 0:~ X_t =\Delta\}$ the lifetime of the process.
Each $\P^x$ is then a probability measure on paths with $X_0 = x$, $\P^x$-a.s for all $x\in E_\Delta$. Furthermore, for $t\geq 0$ let $\theta_t$ be the natural shift operator of the the process, i.e. $\theta_t(X_s(\omega))=X_{t+s}(\omega)$ for all $s \geq 0$.  
 Throughout, we assume that $X$ is a \emph{standard process}, see Definition \ref{def:stdproc}, in particular we assume that $X$ has c\`adl\`ag paths and satisfies the strong Markov property. We write $P= (P_t)_{t\geq 0}$ for the Markovian transition semigroup of $X$ and $U = \int_0^\infty P_t\dd t$ for its potential and write as usual $P_tf$, $\mu P_t$, $Uf$ and $\mu U$ for the actions on Borel functions $f:E\to \R $ and Borel measures $\mu$ on $E$. For an $(\F_t)_{t\geq 0}$-stopping time $T$, we write $P_T(x,\dd y)=\P^x\left(  X_T\in \dd  y; T <\zeta\right)$ and for 
first hitting times $T_A =\inf\{t>0:\,X_t\in A\}$ of A$\in \mathcal{E}$, we write $P_A= P_{T_A}$. We write $A^r$ for the regular points of a nearly Borel set $A$, see Table \ref{tbl:fine} in Appendix \ref{app:basic}.
 
 A central role will be played by lifting $X$ to a space-time process $\overline X$, that is $\overline X_t : = (\tau_t,X_t)$ with $\tau_t=\tau_0+t$, with the space-time semigroup $Q=(Q_t)_{t\geq 0}$ acting on Borel functions $g:\R\times E\to\R$ and Borel measures $\mu$ on $E$ as follows:
$$ Q_s g(t,x)= P_sg(t+s, \cdot)(x) \qquad (\delta_s\times\mu) Q_t (I\times A)= \P^\mu(X_{s+t} \in A)\ind_{\{s+t\in I\}},$$
where $t\in \R$, $x\in E$, $A\in \mathcal{E}$ and $I\subseteq \R$.  

\paragraph{Duality.}

Throughout this paper we make the following assumption,
\begin{assumption} \label{asn:dual}
There exists a standard process $\h X$ with semigroup $\h P$ on the same probability space, and some $\sigma$-finite measure $\xi$ on $E$ such that for all $t \geq 0$ and $f,g \geq 0$ $\mathcal{E}^\ast$ -measurable,
\begin{equation} \label{eq:dual} 
 \int_{E} (P_t f) g \dd\xi = \int_{E} f (g \h{P}_t) \dd\xi. 
\end{equation}
Furthermore, the semigroups of $X$ and $\h X$ are absolutely continuous with respect to $\xi$,
\begin{equation} \label{eq:asnAC}
P_t(x,\cdot) \ll \xi, \quad \h{P}_t(\cdot, y) \ll \xi ,\qquad \forall x,y \in E.
\end{equation}
\end{assumption}

\begin{remark}
  Relation \eqref{eq:dual} is referred to in the literature as \emph{weak duality}.
  The processes $X$ and $\h{X}$ are said to be in \emph{strong duality} with respect to $\xi$ (as defined in \cite[Ch. VI]{blumenthal2007markov} or \cite[Ch.13]{chung2006markov}), if, in addition to \eqref{eq:dual}, the resolvent kernels are absolutely continuous with respect to $\xi$. This is weaker than the absolute continuity of the semigroup, so that in particular, strong duality of $X $ and $\h{X}$ holds under Assumption \ref{asn:dual}.
\end{remark}

We write $\h P$ and $\h U$ for the semigroup and potential kernel of $\h X$
, and we denote the actions of these operators on Borel functions $f$ and measures $\mu$ on the other side as for $X$, i.e. $f\h P_t$, $f\h U$ and $\h U \mu$. 
Furthermore, we use the prefix ``co'' for the corresponding properties relating to $\h X$, e.g.~coexcessive, copolar, cothin, etc., and we write $\h T_A = \inf\{t>0:\,\h X_t\in A\}$ and  ${}^rA$ for the coregular points of a measurable set $A$.

By~\cite{getoor1982excursions,wittmann1986natural}, absolute continuity of the semigroups implies that the corresponding space-time processes  $(\tau_t,X_t)$, and $(\h{\tau}_t, \h{X}_t)$, where  $\h{\tau}_t= \h{\tau}_0-t$, are in strong duality with respect to the measure $\lambda \otimes \xi$, where $\lambda$ is the Lebesgue measure on the real line.
We denote by $\h{Q}$ the semigroup corresponding to the space-time process $(\h{\tau}_t, \h{X}_t)$.
For every $s\geq 0$ and $(\mathcal{B}(\R)\times \mathcal{E})$-$\mathcal{B}(\R)$-measurable function $g$,
\begin{align*}
 (Q_sg)(t,x) = P_sg(t+s,\cdot)(x), \qquad (g\h Q_s )(t,x)= g(t-s, \cdot )\h P_s (x). 
\end{align*}
In addition, there exists a Borel function $(t,x,y) \mapsto p_t(x,y)$ such that for all $t>0$ and $x,y$ in $E$,
$P_t(x,\dd y) = p_t(x,y) \xi(\ddi y)$ and $\h{P}_t(\ddi x,y) = p_t(x,y)\xi(\ddi x)$, and $p$ satisfies the Kolmogorov--Chapman relation
\begin{equation} \label{eq:KC}
\forall t,s >0, \;\; \forall x,y \in E:\; \;\; p_{t+s}(x,y) = \int \xi(\ddi z) p_t(x,z) p_s(z,y).
\end{equation}
 The function $u(x,y) := \int_0^\infty p_t(x,y)\dd t$ is excessive in $x$ (for each fixed $y$), coexcessive in $y$, and is a density for $U$ and $\h{U}$.
 
 Note that the duality assumption implies by \cite[Ch. IV, Prop. (1.11)]{blumenthal2007markov}) that a measure $\mu$ is excessive if and only if it has a density which is coexcessive and finite $\xi$-almost everywhere.
Hence, the density of the potential $\mu U$ with respect to $\xi$ is given by the (coexcessive) potential function $\mu \h U$.

\begin{table}[h!]
\begin{spacing}{1.2}
\begin{tabular}{L{0.23\textwidth}R{0.13\textwidth}@{$\,=\,$}L{0.23\textwidth} R{0.13\textwidth}@{$\,=\,$}L{0.23\textwidth}}
\toprule
& \multicolumn{2}{c}{$X$} & \multicolumn{2}{c}{$\h X$}\\
\midrule
semigroup & $P_tf(x)$ & $\int p_t(x, y) f(y) \xi(\ddi y)$ & $f\h P_t(y) $ & $ \int f(x) p_t(x, y)  \xi(\ddi x)$\\
 & $\mu P_t(\ddi y)$ & $\int \mu(\ddi x) p_t(x, y)  \xi(\ddi y)$ & $\h P_t\mu(\ddi x) $ & $ \int \xi(\ddi x) p_t(x, y) \mu(\ddi y) $\\
potential function & $Uf(x) $ & $ \int u(x,y) f(y) \xi(\ddi y)$ & $f\h U(y) $ & $ \int f(x)u(x,y)  \xi(\ddi x)$\\
& $U\mu (x) $ & $\int u(x,y) \mu(\ddi y) $ & $\mu \h U(y) $ & $ \int \mu (\ddi x) u(x,y)$\\
potential measure   & $\mu U(\ddi y) $ & $\int \mu(\ddi x) u(x,y) \xi(\ddi y)$ & $\h U\mu(\ddi x) $ & $ \int  u(x,y)  \mu(\ddi y)\xi (\ddi x)$\\
\bottomrule
\end{tabular}
\end{spacing}
\caption{(Densities of) semigroup and potentials for $X$ and its dual $\h X$.}
\end{table}

\begin{remark}\label{rem:Borel}
The functions which are (co-)excessive with respect to $P$, $\h P$, $Q$ and $\h Q$ are actually Borel-measurable. 
Indeed, strong duality of the corresponding processes guarantees the existence of a so-called \emph{reference measure}\footnote{a \s-finite measure $\xi$ is a reference measure for $X$ if for all Borel $B$, ($U(x,B) = 0$ for all $x$) $\Leftrightarrow$  $\xi(B)=0$.} (for more details see \cite[Ch. VI]{blumenthal2007markov}). In this case Proposition (1.3) in \cite[Ch. V]{blumenthal2007markov} implies that excessive functions are Borel-measurable. 
\end{remark}

We repeatedly use the following classical result, 

\begin{proposition}[Hunt's switching formula, {\cite[VI.1.16]{blumenthal2007markov}}] \label{prop:hsf}
Let $X$, $\h X$ be standard processes in strong duality. Then for all Borel-measurable $B$, one has $P_B u = u \h{P}_B$, i.e. for all $x, y \in E$,
$$\E^x \big[ u(X_{T_B},y) \big] = \h \E^y \big[ u(x,\h X_{\h T_B}) \big].$$
\end{proposition}

\begin{remark}
The dual process $\h{X}$ can be thought of as $X$ running backwards in time. In fact, strong duality implies that for non-negative bounded Borel functions $f$ and $g$ it holds
\begin{align*}
 \E^\xi[ f(X_0)g(X_t)] = \h \E^\xi[ f(\h X_t) g(\h X_0)].
\end{align*}
More generally and ignoring technicalities (see~\cite[Ch.13]{chung2006markov} for details), if we take $\Omega$ to be the canonical probability space and let $r_t:
\Omega \rightarrow \Omega$ denote the right-continuous time reversal at time $t$, that is
$\omega':=r_t(\omega)$ is given as $\omega'(s):=\omega(t-s-)$, then for any $F$
 that is $\F_t$-measurable 
\begin{align*}
 \E^\xi [ F] = \h \E^\xi [ F\circ r_t]. 
\end{align*}
Informally, strong duality of $X$ with another standard process requires that two conditions are met:
\begin{inparaenum}[(i)]
 \item $X$  admits an excessive reference measure,
   \item the right-continuous version of its time reversal is a standard process and in particular satisfies the strong Markov property.
\end{inparaenum}
We refer to~\cite[Ch. 15]{chung2006markov} and~\cite{smythe1973existence} for a detailed discussion.
\end{remark}

\begin{remark}\label{rem:dirichlet}
A practical approach to obtain Markov processes in duality is via Dirichlet forms.
Given a Markov process with generator $\mathcal{L}$, this consists in considering the bilinear form 
\begin{align*}
\mathcal{D}(f,g) := -  \int (\mathcal{L} f) g \dd\xi,
\end{align*}
extended to a suitable class of functions $f,g$.
The theory of Dirichlet forms, see e.g.~\cite{MaRoeckner}, then provides sufficient analytic criteria on $\mathcal{D}$ so that it is associated to a pair of (standard) Markov processes in weak duality with respect to $\xi$. It is also possible to obtain existence (and further properties) of transition densities for a Markov semigroup by considering functional inequalities (such as Nash inequality) satisfied by the associated Dirichlet form, see e.g. \cite{CKS87}.

\end{remark}

\section{A free boundary characterisation} \label{sec:characterization}
\begin{definition}[Root barrier] A subset $R$ of $\R_+ \times E$ is called a \emph{Root barrier} for $X$ if $R$ is nearly Borel-measurable with respect to the space-time process $\overline X$ and
  \[(t,x) \in R, \;s > t \;\; \Longrightarrow \;\;(s,x) \in R.\]
We call the first hitting time $T_R = \inf \{t > 0 : (t, X_t)\in R\}$ the \emph{Root stopping time} associated with $R$.
\end{definition}
Dealing with the regularity of $R$ is a central theme of this article and it is useful to introduce ``right-''  and ``left-''continuous modifications $R^-$ and $R^+$ of $R$.
\begin{definition}
For a Root barrier $R$ denote with \[R_t=\{ x \in E: (t,x) \in R\}\] the section at time $t$.
 We define $R^- \subset R \subset R^+$ as  
 \begin{align*}
 R^-&=\bigcup_{t\ge 0 }~[t,\infty) \times R^-_t \quad \text{ with } R^-_t = \bigcup_{s<t}~ R_s, \\
 R^+&= \bigcup_{t\ge 0 }~[t,\infty) \times R^+_t \quad \text{ with } R^+_t = \bigcap_{s>t} ~R_s.
 \end{align*}
\end{definition}
\begin{remark}
An equivalent definition of the barrier is that the mapping $t \mapsto R_t$ is non-decreasing. This also implies that $R^-$ and $R^+$ are barriers as well.
As $R$ is nearly Borel-measurable with respect to $\overline{X}$ then so are the shifted barriers $R^{s} := \{ (t-s,x) ,\;\; (t,x) \in R\}$ for any  $s \in \R$, as then 
$$R^- = \bigcup_{s < 0,~ s \in \Q} R^{s}, \;\;\;\;\;\;\;\;\; R^+ = \bigcap _{s>0,~ s \in \Q} R^s.$$
 \end{remark}

\begin{definition}[Balayage order]
Two probability measures $\mu$ and $\nu$ are in \emph{balayage order}, if their potentials $\mu U$ and $\nu U$ satisfy 
\begin{equation}
\mu U(A) \geq \nu U (A) \qquad \text{for all measurable sets } A.\label{eqn:convorder}
\end{equation}
In this case we will write $\mu \prec\nu $ and say that $\mu$ \emph{is before} $\nu$. 
\end{definition}
\begin{remark}
Under Assumption \ref{asn:dual}, \eqref{eqn:convorder} is equivalent to 
\begin{equation}\label{eq:balayage}
\mu \h U (x) \geq \nu \h U (x) \qquad \text{for all } x\in E.
\end{equation}
The inequality \eqref{eq:balayage} holds everywhere if and only if it holds $\xi$-almost everywhere, since both sides are coexcessive functions.
\end{remark}
We now state our main result,
\begin{theorem} \label{thm:main}
  Let $X$ be a Markov process for which Assumption \ref{asn:dual} holds. 
  Let $\mu, \nu$ be two measures such that $\mu U$ and $\nu U$ are $\sigma$-finite measures and such that $\nu$ charges no semipolar set.
  If $\mu \prec \nu $ then there exists a Root barrier $R$ for $X$ such that  
\[\mu P_{T_R} = \nu.\]
Moreover, if we set 
\begin{equation}\label{eq:reprred}
  f^{\mu,\nu}(t,x) := \inf \big\{ g\mbox{ $\h Q$-excessive:} \;\;\; g \geq \mu\h U(x) \ind_{\{t \leq 0\}} + \nu \h U(x) \ind_{\{t >0\}} \big\},
\end{equation}
then
\begin{enumerate}
\item\label{thm:itm:MRE}  $f^{\mu,\nu}(t,x) = \mu P_{t \wedge T_R} \h{U}(x)$,
\item\label{thm:itm:optimal} 
$ T_R = \argmin \limits_{S:~ \mu P_S = \nu} \mu P_{t\wedge S} U(B)$  for any Borel set $B$ and $t\geq 0$,
\item\label{thm:itm:freeBD} in the above we may take
  \begin{align*}
R = \left\{(t,x)\in \R_+ \times E\;\; |\;\; f^{\mu,\nu}(t,x) = \nu \h U (x) \right\}.
  \end{align*}
\end{enumerate}
\end{theorem}
\noindent Besides existence and optimality of a Root stopping time, the main interest of Theorem~\ref{thm:main} is that item~\eqref{thm:itm:freeBD} provides a way to compute the Root barrier for a large class of Markov processes ranging from L\'evy processes to hypo-elliptic diffusions, see the examples in Section~\ref{sec:examples}.
Concretely, it allows to use classical optimal stopping and the dynamic programming algorithm to compute $f^{\mu,\nu}$ and hence $R$.
We state this as a corollary: 
\begin{corollary}\label{cor:OST}
Using the same notation and assumptions as in Theorem~\ref{thm:main} it holds that 
\begin{enumerate}
\item\label{cor:itm:OST}  $f^{\mu,\nu}$ is the value function of the optimal stopping problem
 \begin{equation}
f^{\mu,\nu}(t,x) = \sup_{\tau} \E^x \left[ \mu\h U\left(\h{X}_\tau\right) \ind_{\{\tau=t\}} + \nu\h U\left(\h{X}_\tau\right) \ind_{\{\tau<t\}} \right] \quad\forall t \geq 0, x \in E, 
\end{equation}
where the supremum is taken over stopping times $\tau$ taking values in $[0,t]$.
\item\label{cor:itm:DP} 
If we define for $n \geq 0$ the function $f_n^{\mu,\nu}$ on $\{ k2^{-n},~ k \geq 0\}  \times E$ by
\[f_n^{\mu,\nu}(0,\cdot) = \mu \h{U}, \;\;\;\;f_n^{\mu,\nu}(2^{-n}(k+1),\cdot) = \max \left\{ \left(f^{\mu,\nu}_n(2^{-n}k,\cdot) \right)\h{P}_{2^{-n}} ,\nu \h{U}\right\},\]
then for each $t\geq 0$, $x \in E$,
\[f^{\mu,\nu}(t,x) = \lim_{n \to \infty} f^{\mu,\nu}_n\left( 2^{-n} \lfloor 2^n t \rfloor,x\right).\]

\end{enumerate}
\end{corollary}
\noindent 
Informally, $f^{\mu,\nu}$ is the solution of the  obstacle problem
\begin{align}\label{eq:obsPDE}
  u(0,\cdot) = \mu\h U, \quad \min\left[(\partial_t - \L)u, u - \nu \h U \right] = 0 \qquad \mbox{ on } (0,+\infty) \times E,
\end{align}
where $\L$ is the generator of the dual process $\h X$.
However, to make this rigorous is in general a subtle topic since the obstacle introduces singularities.
Several notions of generalised PDE solutions ranging from variational inequalities to viscosity solutions address this, often together with numerical schemes \cite{barles1991convergence,jakobsen2005continuous,kinderlehrer1980introduction,petrosyan2012regularity}.
This PDE approach to Root's barrier has been carried out in \cite{cox2013root,gassiat2015root} for one-dimensional diffusions.
However, already in the one-dimensional case when the operator involves non-local terms as is the case for many Markov processes, the well-posedness of such obstacle PDEs is an active research area; see e.g.~\cite{barrios2018free,caffarelli2017obstacle}. 
In general, this PDE approach requires stronger assumptions than Assumption~\ref{asn:dual} for the well-posedness of \eqref{eq:obsPDE}; in stark contrast, Corollary~\ref{cor:OST} holds in full generality of Theorem~\ref{thm:main}.

\begin{remark}[Minimal residual expectation]
Item~\eqref{thm:itm:optimal} of Theorem \ref{thm:main} was named \emph{minimal residual expectation} by Rost \cite{rost1976} with respect to $\nu = \mu P_{T_R}$. It implies that  
 $$T_R = \argmin \limits_{S:~ \mu P_S = \nu} \E^\mu [F(S)], \; \mbox{ for any  non-decreasing convex function $F$}.$$
 This is actually an equivalent formulation of the minimal residual expectation property as soon as $(\mu U-\nu U)(E)$ is finite as then this quantity is equal to $\E^\mu[S]$ for all solutions $S$ of \sep. Furthermore, Rost proved in \cite{rost1976} that any stopping time $S$ which is of minimal residual expectation with respect to $\mu P_{T_R}$ necessarily satisfies $S=T_R$ $\P^\mu$-a.s.
\end{remark}

\begin{remark}[Recurrent Markov processes]
 That $\mu U$ and $\nu U$ are \s-finite is a kind of transience assumption, and is usual in this context \cite{falkner1991stopping, rost1976}.
  In the case of one-dimensional Brownian motion or diffusions it is not necessary, see \cite{cox2013root,gassiat2015root}.
  We expect that our result could be extended to the recurrent case (at least in some special cases), but this would require a certain amount of work, see e.g.~\cite{falkner1981distribution} for results for two-dimensional Brownian motion.
\end{remark}
\begin{remark}[Assumptions of Theorem \ref{thm:main}]
From the counterexamples discussed in \cite{falkner1981distribution,falkner1991stopping}, to obtain solutions to {\sep} as non-randomised stopping times, one needs to make:
\begin{enumerate}[(1)]
\item an assumption on the process in order to avoid ``deterministic portions'' in the trajectory. In our case, this is reflected in the assumption of absolute continuity \eqref{eq:asnAC}. This assumption is rather strong but can often be checked in practice. In the case of diffusions, the celebrated H\"ormander's criterion \cite{hormander1967hypoelliptic} gives a simple condition to ensure existence of transition densities with respect to Lebesgue measure. For jump-diffusions, there are also many results providing sufficient criteria for absolute continuity, see for instance \cite{bichteler1987stochastic, picard1996existence}.
\item an assumption on the ``small'' sets charged by initial and target measures (to avoid issues as in the case of multidimensional Brownian motion and Dirac masses). This is why we assume that $\nu$ charges no semipolar sets. Without this assumption, it is not true that there exists a solution to \sep\,  as hitting time of a barrier, or even as an non-randomised stopping time.
In the case where all semipolar sets are polar, following \cite{falkner1983stopped}, we can replace the assumption that $\nu$ charges no (semi)polar set by the assumption that
\begin{equation} \label{eq:falk}
\mbox{there exists a (universally measurable) set $C$ s.t. }\nu(Z) = \mu(Z \cap C), \:\: \mbox{ for all polar }Z.
\end{equation}
Indeed, there exists then a polar set $M \subset C$, and a measure $\gamma$ supported on $M$, $\mu', \nu'$ supported on $M^c$ with $\mu= \mu' + \gamma$, $\nu=\nu' + \gamma$, and $\nu'$ charges no polar sets (cf. \cite[p.50]{falkner1983stopped}). Letting $R'$ be a barrier embedding $\nu'$ into $\mu'$ as given by Theorem \ref{thm:main}, let $R:=R' \cup (\R_+ \times M)$, then $T:= \inf \{t \geq 0, \; (t,X_t) \in R\}$ embeds $\nu$ into $\mu$. 
In \cite{falkner1983stopped} is proven that (if semipolar sets are polar), \eqref{eq:falk} is a necessary condition for a non-randomised solution to \sep\, to exist (in the case where $\mu U \geq \nu U$ but \eqref{eq:falk} does not hold, randomisation of the stopping time at time $0$ is necessary).
\end{enumerate}

\end{remark}

\section{Proof of Theorem~\ref{thm:main}}\label{sec:proof}

The proof of our main result, Theorem \ref{thm:main}, is split into two parts:
\begin{description}
\item[Existence.] 
  We first show that a Root barrier $R$ exists such that $\mu P_{T_R}=\nu$ and that items~\eqref{thm:itm:MRE} and \eqref{thm:itm:optimal} of Theorem~\ref{thm:main} hold.
  Here we rely on classic work of Rost, \cite{rost1976}, that shows that \sep\, has as solution stopping time $T$ that lies between the hitting times of two barriers which differ only by a space-time graph.
We show that these hitting times are necessarily equal; a similar approach was already followed in \cite{beiglbock2017optimal, chacon1986barrier, gassiat2015root} under different assumptions.

\item[Free boundary characterisation.]
  We show item~\eqref{thm:itm:freeBD} of Theorem~\ref{thm:main}, that is that one can take the contact set of the obstacle problem \eqref{eq:obsPDE} as the Root barrier.
  From a conceptual point of view, this is similar to the case of one-dimensional diffusions as studied with PDE methods in \cite{cox2013root,gassiat2015root}.
  However, there the analysis is greatly simplified due to the existence of local times.
  Since local times are not available in our setting, the situation becomes more delicate and requires the analysis of negligible sets via potential theory.
\end{description}

\subsection{Existence}

We prepare the proof of existence and optimality with two lemmas.
The first lemma shows right-continuity of the semigroup when applied to bounded Borel-measurable functions.
\begin{lemma}\label{lem:rightcont semigroup}
Under Assumption \ref{asn:dual}, it holds for all Borel-measurable and bounded functions $f$, for all $x\in E$ and $t>0$
\begin{equation}
\lim_{h\downarrow 0} P_{t+h}f(x) = P_t f(x). \label{eqn:ContWeak} 
\end{equation}
\end{lemma}
\begin{proof}
First, note that if $f$ is continuous then by a.s. right continuity of $t\mapsto X_t$ it is clear that $P_t f$ is right continuous as a function of $t$.

Let $p_t^x = p_t(x,\cdot)$. Since $\int_E p_t^x(y)\xi(\ddi y) = 1$, by de La Vall\'ee Poussin's theorem (see e.g. \cite[Thm. II.22]{dellacherie1966probabilities} \footnote{The de La Vall\'ee Poussin's theorem in literature is given in finite measure spaces. That $\xi$ is infinite is clearly not a problem here. If necessary consider the finite measure $f(y)\wedge 1 \xi(\ddi y)$, applying the theorem to that measure gives that for some superlinear $G$, $\int G(f(y) \vee 1) f(y) \wedge 1 \xi(\ddi y)<+\infty$. 
} there exists a function $G$ which is strictly convex and superlinear (i.e. $\lim_{x \to +\infty} G(x)/x = +\infty$) such that 
 \begin{equation}
\int G(p_t^x(y))\xi(\ddi y) <\infty.\label{eq:integrable}
 \end{equation}
Then for all $h \geq 0$ one has 
\begin{align*}
\int G(p_{t+h}^x(y)) \xi(\ddi y) &= \int G\left( \int p_t^x(z) p_h^z(y) \xi(\ddi z)\right) \xi(\ddi y) \\
&\leq \int \int p_h^z(y)  \xi(\ddi z) G\left(p_t^x(z)\right) \xi(\ddi y) = \int G(p_{t}^x(z)) \xi(\ddi z),
\end{align*}
where we first used Kolmogorov-Chapman's equality \eqref{eq:KC}, then Jensen's inequality and that it holds $\int p_h^z(y) \xi(\ddi z) = 1$ by duality. 
Since $\xi$ is $\sigma$-finite, there exists a countable increasing family of open sets $(E_n)_{n\in \N}$ such that $\medcup_{n\in \N}E_n= E$ and $\xi(E_n)<\infty$ for all $n\in \N$. 

Now fix $n\in\N$. On $E_n$ the integrability condition as in \eqref{eq:integrable} is satisfied for all functions in the family $(p_s^x)_{s \geq t}$. By the de La Vall\'ee Poussin's theorem this is equivalent to $(p_s^x)_{s \geq t}$ being uniformly integrable in $L^1(E_n, \xi)$.

Then, by the Dunford-Pettis theorem (see e.g. \cite[Thm. II.23]{dellacherie1966probabilities}), uniform integrability of $(p_s^x)_{s\geq t}$ implies that it is weakly (relatively) compact in the finite measure space $L^1(E_n, \xi)$. By a diagonal argument there exists a subsequence $s_k \downarrow t$ and a measurable function $q$ such that for all $n$, for all bounded and measurable $f$ one has $\int_{E_n} p^x_{s_k} f \dd\xi \to \int_{E_n} q f \dd\xi$ for $k\to\infty$. 
If we take $f$ as a continuous function supported in $E_n$, by right-continuity of the sample paths, we obtain that $q = p_t^x$. In addition, since $E_n^c$ is closed, by a.s. right-continuity of $X$, one has that
$$\limsup_{k\to\infty} P_{s_k}(x,E_n^c) \leq P_t(x,E_n^c).$$
Hence if $f$ is measurable and bounded by $1$, 
$$\limsup_{k\to\infty} \left| \int p^x_{s_k}f \dd\xi-\int p^x_t f \dd\xi \right| \leq \limsup_{k\to\infty} \left| \int_{E_n^c} p^x_{s_k}f \dd\xi - \int_{E_n^c} p^x_t f \dd\xi \right| \leq 2 P_t(x,E_n^c).$$
Letting $n \to \infty$, the right-hand side goes to $0$ by dominated convergence. Hence $p^x_{s_k}$ converges weakly in $L^1(E,\xi)$ to $p^x_t$. We can use the same line of argument for every subsequence of any sequence $s_k\downarrow t$ to argue the convergence of a subsubsequence. Therefore for all $x\in E$ we have that $p^x_{s}$ converges weakly in $L^1(E,\xi)$ to $p^ x_t$ for $s\downarrow t$ which leads to the required statement. 

\end{proof}
The second Lemma revisits Chacon's idea of ``shaking the barrier'', see also~\cite{beiglbock2017optimal, chacon1986barrier} for similar statements under slightly stronger assumptions.
\begin{lemma} \label{lem:RR+}
If the semigroup of a Markov process satisfies \eqref{eqn:ContWeak}, then for all Root barriers $R$ one has almost surely
\begin{equation}
T_R = T_{R^+}=T_{R^-}.
\end{equation}
\end{lemma}

\begin{proof}
Firstly, by replacing $R$ with $R^+$ if necessary, it is enough to show that $T_R = T_{R^{-}}$ almost surely.
Secondly, if we define \[R(\delta) := R \cap ([\delta,+\infty) \times E),\] we have $T_R = \inf_{\delta>0} T_{R(\delta)}$.
Put together, this implies that it is sufficient to show that for all $\delta>0$, $T_{R(\delta)}=T_{R^-(\delta)}$ $\P^\mu$-a.s.~and below we assume that $R = R(\delta)$ for a given $\delta>0$.

For $\epsilon \in \R$ define 
\begin{align}
 R^\epsilon :=  \bigcup_{t\geq \max(-\epsilon, 0) } [t,\infty) \times R_{t+\epsilon}.
\end{align}
That is, $R^\epsilon$ is the barrier that arises by shifting $R$ in time to the left if $\epsilon>0$ [resp.~to the right if $\epsilon<0$]. 
Now since $R=R(\delta)$,
\[T_R = T_{R^{\delta}} \circ \theta_{\delta} + \delta\]
and for any $0< \epsilon < \delta$ we also have 
\[T_{R^{-\epsilon}} =  T_{R^\delta} \circ \theta_{\delta+\epsilon} + (\delta+\epsilon).\]

Now set $f(x) := \E^x \left[ \exp\left(-T_{R^{\delta}}\right)\right]$ and use the above identities to deduce that for every $0 < \epsilon < \delta$ and every $x$,
 \[\E^{x} \left[ \exp\left(-T_{R}\right)\right] = e^{-\delta} P_\delta f(x), \;\;\;\E^{x} \left[ \exp\left(-T_{R^{-\epsilon}}\right)\right] = e^{-(\delta+\epsilon)} P_{\delta+\epsilon} f(x).\]
From the right-continuity of the semigroup, Lemma \ref{lem:rightcont semigroup}, it follows that
 \[\lim_{\epsilon \downarrow 0} \E^{x} \left[ \exp\left(-T_{R^{-\epsilon}}\right)\right] = \E^{x} \left[ \exp\left(-T_{R}\right)\right].\]
But since $T_{R^{-\epsilon}}\geq T_R$ $\P^x$-a.s.~for all $x$ and for all $\epsilon >0$, this already implies that 
 \[\lim_{\epsilon \downarrow 0} T_{R^{-\epsilon}} = T_R\quad \P^x\text{-a.s.}\]
and we conclude that $T_R = T_{R^-}$ since $R^- = \bigcup_{\epsilon>0} R^{-\epsilon}.$ 
\end{proof}

For the proof of existence and optimality, we rely on the following result obtained by Rost:
\begin{theorem}[Rost, Theorems 1 and 3 in {\cite{rost1976}}]\label{thm:rost}
If $\mu \prec \nu$, then there exists a (possibly randomised) stopping time $T$ which is of minimal residual expectation with respect to $\nu$, i.e. $\mu P_T=\nu$ and 
\begin{equation}
 T = \argmin \limits_{S:~ \mu P_S = \nu} \mu P_{t\wedge S} U(B) \qquad \text{for any Borel set $B$ and }t\geq 0.\label{eq:mre}
\end{equation}
In addition, the measure
$$ \dd t \otimes  ( \mu P_{t \wedge T} U (\ddi x))$$
is given by the $Q$-r\' eduite of the measure 
\[ \dd t \otimes \left( \mu U(\ddi x)\ind_{\{t \leq 0\}}  + \nu U(\ddi x)\ind_{\{t > 0\}}  \right).\]
Furthermore, there exists a finely closed Root barrier $R$ such that $\P^\mu$-a.s.:
\begin{enumerate}
\item $T \leq T_R:= \inf \left\{ t > 0, \; X_t \in R_{t}\right\} $,\label{thm:itm:TTR}
\item $X_{T} \in R_{T+}$.\label{thm:itm:XTRT+}
\end{enumerate}
\end{theorem}

The key ingredients in the proof of Rost's theorem are the filling scheme from \cite{rost1971}, which allows to obtain the existence of $T$ satisfying the optimality property \eqref{eq:mre}, and then a paths-swapping argument (see \cite{hobson2011skorokhod} for a heuristic description), which shows that $T$ is almost the hitting time of a Root barrier (i.e. \eqref{thm:itm:TTR} and \eqref{thm:itm:XTRT+} above). However, this does not imply that $T$ is the hitting time of a Root barrier. 
%
%

In order to conclude item \eqref{thm:itm:MRE} from Theorem \ref{thm:main}, we first see that Lemma \ref{lem:RedMtoF} yields
\begin{equation}\label{eq:equivae}
f^{\mu,\nu} (t,x)= \mu P_{t \wedge T} \h{U}(x), \;\;\;\;\dd t \otimes \xi(\ddi x) \mbox{-a.e.}
\end{equation}
We will prove in Lemma \ref{lem:Properties} that $f^{\mu,\nu}$ is $\h Q$-excessive. Therefore, if we show that $g(t,x):= \mu P_{t\wedge T}\h U(x)$ is $\h Q$-excessive then \eqref{eq:equivae} holds everywhere. For this we need to show that $g$ satisfies $g \h{Q}_t \to g$ as $t \to 0$. But this follows from the definition since
\[\liminf_{t \to 0} \, (g \h{Q}_t)(s,\cdot) =\liminf_{t \to 0} g(t-s, \cdot)\h P_t =  \liminf_{t \to 0}\mu P_{(s-t) \wedge T} \h{U} \h{P}_t \geq \liminf_{t \to 0} \mu P_{s \wedge T} \h{U} \h{P}_t = \mu P_{s \wedge T} \h{U}.\]
Secondly, from $f^{\mu,\nu} (t,x)= \mu P_{t \wedge T} \h{U}(x)$, it also follows from Theorem 1 in \cite{rost1976} that $T$ is the unique stopping time minimising $\mu P_{t\wedge T }\h U$ for all $t\geq 0$ among all stopping times embedding $\nu$ in $\mu$.

Furthermore, $X_{T} \in R_{T+}$ in \eqref{thm:itm:XTRT+} from Theorem \ref{thm:rost} by Rost implies 
$T \geq  T_{R^+} :=\inf \left\{ t > 0, \; X_t \in R_{t+}\right\}$ on $\{T>0\}$. If $T=0$ we have $X_0 \in R_{0+}$ and if $X_0 \in R_{0+}^r$ then $T_{R^+}=0$, so that combined we get
\begin{equation*}
\P^\mu(T=0<T_{R^+}) \; \leq \P^\mu(X_T \in R_{0+} \setminus R_{0+}^r) \; =\; \nu(R_{0+} \setminus R_{0+}^r) \;= \; 0,
\end{equation*}
where we used that $R_{0+} \setminus R_{0+}^r$ is semipolar and that by assumption $\nu$ charges no semipolar sets. Hence combined with item \eqref{thm:itm:XTRT+} from Theorem \ref{thm:rost}, one has $T_{R^+} \leq T \leq T_R$, and we can conclude the existence of a solution satisfying items \eqref{thm:itm:MRE} and \eqref{thm:itm:optimal} in Theorem \ref{thm:main} with Lemma \ref{lem:RR+}.

\subsection{Free boundary characterisation}

Let $T=T_R$ be the unique Root stopping time solving \sep\, from the previous section with the respective Root barrier $R$.
We want to prove $T=\tilde{T}$ with $\tilde{T} := T_{\tilde{R}}$, where $\tilde{R}$ is defined as in Theorem \ref{thm:main}
\[\tilde{R} := \left\{ (t,x) \in \R\times E~|~f^{\mu,\nu} (t,x) = \nu \h U (x)\right\}.\]
The proof is split into two inequalities given in Proposition \ref{prop:leq} and Proposition \ref{prop:geq}.
First, we show some useful properties of the Root barrier $\tilde R$:

\begin{lemma}\label{lem:Properties} The function $f^{\mu,\nu}$ and the resulting Root barrier $\tilde{R}$ satisfy the following properties:
\begin{enumerate}
\item\label{itm:q exc} $f^{\mu,\nu}$ is $\h Q$-excessive and non-increasing in $t$,
\item \label{itm: fclosed}$\tilde{R}$ is a Borel-measurable and $\h Q$-finely closed Root barrier.
\end{enumerate}
\end{lemma}

\begin{proof} For~\eqref{itm:q exc}, note that any function $f:\R\times E\to \R$ is $\h Q$-finely continuous if and only if the process $t\mapsto f(s-t, \h X_t)$ is $\P^{\delta_{(s,x)}}$-a.s. right continuous for all $(s, x)\in \R\times E$ (see e.g. \cite[Theorem (4.8)]{blumenthal2007markov}) .
As in the obstacle $h(t,x) =\mu\h U(x) \ind_{\{t \leq 0\}} + \nu \h U(x) \ind_{\{t >0\}}$, we have the $\h P$-finely continuous functions $\mu \h U$ and $\nu\h U$ making $t\mapsto\mu\h U(\h X_t)$ (when $t< s$) and $t\mapsto\nu\h U(\h X_t)$ (when $t> s$) $\P^{\delta_x}$-a.s. right-continuous for all $x\in E$. Furthermore, $t\mapsto h(s-t,x)$ is right continuous at $t=s$ which together makes $h$ $\h Q$-finely continuous. By Proposition \ref{prop:red1} it then follows that $f^{\mu,\nu}$ is $\h Q$-excessive.
 Further, $f^{\mu,\nu}$ is non-increasing in $t$ since $h$ is non-increasing.

For~\eqref{itm: fclosed}, note that the $\h Q$-excessive function $f^{\mu,\nu}$ is Borel-measurable, see Remark \ref{rem:Borel}, and the barrier $\tilde{R}$ is a level set of the Borel-measurable function $(t,x) \mapsto f^{\mu,\nu}(t,x) -\nu \h U(x)$, hence it is Borel-measurable. 
Therefore $\tilde{R}$ is $\h Q$-finely closed since it is the set where the two finely-continuous functions $f^{\mu,\nu}$ and $h$ coincide, and it is a barrier by time monotonicity of $f^{\mu,\nu}$.

\end{proof}

\begin{proposition}\label{prop:leq}
$\tilde{T} \leq T$.
\end{proposition}

\begin{proof}
Since $\tilde{T} = T_{\tilde{R}} = T_{\tilde{R}^+}$ by Lemma \ref{lem:RR+}, we only need to prove $T_{\tilde{R}^+}\leq T$. 
Since $\mu U$ is $\sigma$-finite, $\mathcal{N} = \{ \mu\h U = \infty \}$ is polar (cf.~\cite[(3.5)]{blumenthal2007markov}).
Let $(s,y)$ be such that $y \in {}^r R_{s}$ and $y \notin \mathcal{N}$.
One has
\begin{align}
0\leq \mu P_{s\wedge T}\h U(y) - \mu P_T\h U(y) 
&=\E^\mu\left[\ind_{\{s\leq T\}}\big(u(X_s, y) -u(X_T, y)\big)\right]\label{eqn:leq1}
\end{align}
and since $T\leq s + T_{R_s}\circ \theta_s$ on $\{s\leq T \}$ as $s\mapsto R_s$ is non-decreasing, we can apply the Markov property to obtain
\begin{align*}
\eqref{eqn:leq1} &\leq \E^\mu\left[\ind_{\{s\leq T\}}\big(u(X_s, y) -P_{R_s}u(X_s, y)\big)\right].
\end{align*}
By the switching identity (Proposition \ref{prop:hsf}) and since $y\in {}^rR_s$ we have
\begin{align}
P_{R_s}u(x, y) = \E^x\big[u(X_{T_{R_s}}, y)\big]= \h\E^y\big[u(x,\widehat X_{\widehat T_{R_s}})\big] =u(x,y)
\end{align}
for all $x \in E$ and hence 
\begin{equation}
\mu P_{s\wedge T}\h U(y) = \mu P_T\h U(y) =\nu\h U(y).
\end{equation}
Thus, we can conclude that $(s,y) \in \tilde{R}$.

Now for any $\varepsilon >0$, if $t <q < t+\varepsilon$, $R_t \setminus {}^r R_{t+\varepsilon} \subset R_q  \setminus  {}^r R_q$. Since $\bigcup_{q \in \Q} R_q  \setminus  {}^r R_q$ is semipolar, and since $\nu$ charges no semipolar sets, it follows that a.s. $X_T \in {}^r R_{T+\varepsilon} \setminus \mathcal{N}$. By the previous paragraph, this means that $X_T \in \bigcap_{\varepsilon >0} \tilde{R}_{T+\varepsilon} = \tilde{R}_{T+}$. Hence $T_{\tilde{R}^+} \leq T$.
\end{proof}  

Before we prove the inverse inequality, we first need a preliminary lemma:

\begin{lemma} \label{lem:TTR2}
Assume that for some measure $\eta$ which charges no semipolar sets, some stopping time $\tau$ and some nearly Borel-measurable set $A$ one has $\eta \h{U} = \eta P_\tau \h{U}$ on $A$. Then $\tau \leq T_A$, $\P^\eta$ almost surely.
\end{lemma}

\begin{proof}
We first write
\[
\eta P_\tau \h{U} \geq \eta P_\tau \h{U}  \h{P}_A = \eta  \h{U}  \h{P}_A    = \eta P_A \h{U},
\]
where we have used in the first inequality that $\eta P_\tau \h{U} $ is coexcessive and in the following equality, that the coexcessive functions $\eta  \h{U}$ and $\eta  P_\tau \h{U}$ coincide on $A$ and therefore also on its cofine closure on which $\h{P}_A$ is supported. The last equality follows by the switching identity.

Therefore it holds that $ \eta P_A U \leq  \eta P_\tau U $, i.e.~the measures $\eta P_A$ and $\eta P_\tau$ are in balayage order.
We then follow the proof of \cite[Lemma p.8]{rost1976}. By \cite{rost1971}, since $\eta P_A \succ \eta P_\tau$, there exists a stopping time $\tau'$ (possibly on an enlarged probability space) which is later than $\tau$ such that the process arrives in the measure $\eta P_A$ at time $\tau'$, i.e. $\tau' \geq \tau$ $\P^\eta$-a.s. and $\eta P_{\tau'} = \eta P_A$. We can assume without loss of generality that $A$ is finely closed and then this implies that $\P^\eta (X_{\tau'}\in A)=1$. In particular, if $D_A := \inf\{t\geq 0,~ X_t \in A\}$ 
, then we have $\tau' \geq D_A$ $\P^\eta$-a.s. However, since $\eta(A \setminus A^r) = 0$ it holds that $T_A=D_A$ $\P^\eta$-a.s., so that $\tau' \geq T_A$. Since $\tau' > T_A$ would be a contradiction to $\eta P_{\tau'}U = \eta P_A U$, we conclude that $T_A=\tau'$, and therefore $T_A \geq \tau$ $\P^\eta$-almost surely.\end{proof}

\begin{proposition}\label{prop:geq}
$\tilde{T} \geq T$.
\end{proposition}

\begin{proof}
  We first show that for all $t \in \Q_+:=\Q \cap (0,+\infty)$, one has $T \leq t + T_{\tilde{R}_t} \circ \theta_t$.
  For this we first prove
\begin{equation} \label{eq:shifted}
\mu P_{T} = \mu P_{t \wedge T} P_{T^{t}}
\end{equation}
where for fixed $t \in \Q_+$ the stopping time $T^{t}:=\inf \{ s >0: X_s \in {R}_{t+s}\}$ is the hitting time of $R$ shifted in time by $t$. This holds since for all Borel-measurable functions $f$ it holds
\begin{align*}
\mu P_{t \wedge T} P_{T^{t}} f &= \E^{\mu} \left[ f(X_{t+T^{t}\circ \theta_t}) \ind_{\{t < T\}} +  P_{T^{t}} f(X_T) \ind_{\{t \geq T\}} \right] .
\end{align*}
Since $T =t+T^{t}\circ \theta_t$ on $\{ t < T\}$, it holds that $\E^{\mu} [ f(X_{T^{t}\circ \theta_t}) \ind_{\{t < T\}}] = \E^{\mu} [ f(X_{T}) \ind_{\{t < T\}}]$. Furthermore, we know that by definition of $T^t$ we have $P_{T^{t}} f = f$ on ${R}_t^r$, since $t\mapsto R_t$ is non-decreasing. As $\P^{\mu}(X_T \in {R}_t \setminus {R}_t^r) = \nu({R}_t \setminus {R}_t^r) =0$ since $\nu$ does not charge semipolar sets, it holds that $\E^ x[P_{T^{t}} f(X_T) \ind_{\{t \geq T\}}] = \E^x[f(X_T) \ind_{\{t \geq T\}}]$. Together this implies $\mu P_{t \wedge T} P_{T^{t}} f   = \mu P_Tf$.

Secondly, note that $\mu P_{t \wedge T} \ll \xi + \nu$ does not charge semipolar sets. Since $\mu P_T \h{U} = \mu P_{t \wedge T} \h{U}$ on $\tilde{R}_t$, we can choose $\eta=\mu P_{t\wedge T}$, $\tau= T^{t}$ and $A = \tilde{R}_t$ in Lemma \ref{lem:TTR2} to obtain that $T^t \leq T_{\tilde{R}_t}$, $\P^{ \mu P_{t \wedge T}}$-a.s. We write
\[ \P^\mu( T > t + T_{\tilde{R}_t} \circ \theta_t) = \E^{\mu}\left[ \ind_{\{t<T\}} \P^{X_t} \left(T^t > T_{\tilde{R}_t}\right)\right] = 0. \]
and this implies
$$\P^\mu\big( \exists t \in \Q_+:~ X_s \in \tilde{R}_t \text{ for some } s\in[t, T)\big) = 0.$$
Since 
$$\tilde{R}_s^- \subset \bigcup_{t \leq s, t \in \Q_+} \tilde{R}_t$$ 
this implies that $T_{\tilde{R}^-} \geq T$ $\P^\mu$-almost surely, which concludes the proof by Lemma \ref{lem:RR+}.
\end{proof}

\section{Examples}\label{sec:examples}
In this section we apply Theorem~\ref{thm:main} to concrete Markov processes.
The examples are
\begin{description}
\item[Continuous-time Markov chains.]
  This is a toy example but we find it instructive since many abstract quantities from potential theory become very concrete and simple; e.g.~the obstacle PDE reduces to a system of ordinary differential equations.
 \item[Hypo-elliptic diffusions.]
   This is a large and important class of processes. 
   In the one-dimensional case we recover the setting of \cite{cox2013root,gassiat2015root} but for the multi-dimensional case the results are new to our knowledge.
   As concrete example we give a Skorokhod embedding for two-dimensional Brownian motion and Brownian motion in a Lie group. 
 \item[$\alpha$-stable L\'evy processes.] There is very little literature on the Skorokhod embedding problem for L\'evy processes, see \cite{doering2017skorokhod} for references.
   We apply our results to $\alpha$-stable L\'evy processes which are of growing interest in financial modelling, see e.g. \cite{tankov2003financial}, as they are characterised uniquely as the class of L\'evy processes possessing the self-similarity property.  
  Due to the infinite jump-activity such processes are hard to analyse but potential theoretic tools are classic in this context and much is known about their potentials, see \cite{bertoin1998levy, blumenthal1960some,bogdan2009potential,kyprianou2014potentials}. 
\end{description}
Two remarks are in order: firstly, the question to characterise or even construct measures $\mu,\nu$ that are in balayage order $\mu \prec \nu$ for a given Markov process seems to be a difficult topic.
In the case of one-dimensional Brownian motion this reduces to the convex order which is usually easy to verify but already for multi-dimensional Brownian motion it can be (numerically) difficult to check if two given measures are in balayage order. 
Secondly, we reiterate the discussion after Corollary~\ref{cor:OST} that the PDE formulation usually requires stronger assumptions whereas the discrete dynamic programming algorithm, Corollary~\ref{cor:OST}, applies to Theorem~\ref{thm:main} in full generality.
All our examples were computed using the dynamic programming equation stated as item~\eqref{cor:itm:DP} in Corollary~\ref{cor:OST}.
\subsection{Continuous-time Markov chains}
Let $Y=(Y_n)_{n\in \N}$ be a discrete-time Markov chain on a discrete state space 
 $E\subset\Z$ and transition matrix $\Pi$ such that $\Pi(x,y) = q(y-x)$ for all $x,y\in E$ and a probability measure $q$. Imposing $\exp(\lambda)$-distributed waiting times at each state, we arrive at the continuous-time Markov chain $X=(X_t)_{t\geq 0}$ with transition function
\begin{equation}
p_t(x,y) = \e^{-\lambda t}\sum_{k=0}^ \infty \frac{(\lambda t)^k}{k!}\Pi^k(x,y).
\end{equation}
The process $X$ is dual to the continuous-time Markov chain $\h X$ with transition matrix $\h \Pi = \Pi^T$ and the same transition rate $\lambda$ at each state with respect to the counting measure. The potentials are given with respect to the function
\begin{equation}
u(x,y) = \sum_{k=0}^ \infty \Pi^k(x,y)
\end{equation}
and the potential function of a measure $\mu$ is given by
\begin{equation}
\mu \h U (y) = \sum_{k=0}^ \infty \sum_{x\in E} \Pi^k(x,y)\mu ( x).
\end{equation}

\begin{example}[Asymmetric random walk on $\Z$]
Let $Y$ be the asymmetric random walk on $\Z$, that is $\Pi(x, x+1)= p\in (\frac{1}{2}, 1]$ and $\Pi(x, x-1)= 1-p=:q$.
  By translation invariance and a standard result (see e.g. \cite{revesz2005random}) it then holds for the potential kernel of $X$
\begin{equation}
u(x,y) = u(0, y-x) = \begin{cases}
\frac{1}{p-q}, & y\geq x,\\
\frac{1}{p-q}\cdot \left(\frac{p}{q}\right)^{y-x}, & y\leq x.
\end{cases}
\end{equation}
Now let $\mu = \delta_0$ and $\nu=\sum_{l=1}^N a_l\delta_{x_l}$ for some $N\in \N$, $a_l>0$, $\sum_{l=1}^N a_l=1$ and $0<x_1<\dots <x_N$. Then 
\begin{align}
\nu \h U(y) = \sum_{l=1}^N a_lu(x_l, y) = \begin{cases}
\frac{1}{p-q}\cdot \sum_{l=1}^N a_l\cdot  \left(\frac{p}{q}\right)^{y-x_l}, & y\leq x_1\\
\frac{1}{p-q}\cdot \left[ \sum_{l=1}^Ka_l  +\sum_{l=K+1}^Na_l\cdot \left(\frac{p}{q}\right)^{y-x_l}\right], & x_K <y\leq x_{K+1},~ 1\leq K\leq N-1,\\
\frac{1}{p-q}, & y\geq x_N.
\end{cases}
\end{align}
Since $p>q$, we have $\nu \h U\leq \mu \h U$ for all such $\nu$. The generator of $\h X$ is given by
\begin{equation}
\h \cL f(y) = \lambda \cdot[pf(y-1) + qf(y+1) -f(y)]
\end{equation}
and the obstacle problem \eqref{eq:obsPDE} reduces to the following set of ODEs:
\begin{align*}
u(0, x) &= \mu \h U(x),\\
\partial _t u(t, x) &= \begin{cases} \lambda \cdot[pu(t,x-1) + qu(t,x+1) -u(t,x)]  & \mbox{ if } u(t,x) > \nu \h U(x),  \\ 0 & \mbox{ if } u(t,x) = \nu \h U(x).  \end{cases}
\end{align*}
Then either classical methods for solving this set of coupled ODEs can be applied or we can directly apply the dynamical programming approach as in Corollary \ref{cor:OST} as follows: 
For $\epsilon>0$ small enough, we choose $x_0$ such that $\mu \h U(x_0)-\nu \h U(x_0)< \epsilon$. We approximate the function $f^{\mu, \nu}$ on the set $\{x_0, x_0+1, \dots , x_N\}$ at discrete time points $t_k=\frac{k}{2^n}$ for fixed $n$:
\begin{align*}
f_n^{\mu,\nu}(0,y) &= \mu \h{U}(y),\\
f_n^{\mu,\nu}(t_k,y) &= \mu \h{U}(y) \qquad\text {for } y<x_0 \text{ or } y\geq x_N,\\
f_n^{\mu,\nu}(t_{k+1}, y) &= \max\big\{ (1- \sfrac{\lambda}{2^{n}} )f_n^{\mu,\nu}(t_k, y) +\sfrac{\lambda}{2^{n}} \big (pf_n^{\mu,\nu}(t_k, y-1)+qf_n^{\mu,\nu}(t_k, y+1)\big), \nu \h U(y)\big\}.
\end{align*}
For example, we take $\lambda=1$, $p=\frac{2}{3}$ and $\nu = \frac{1}{4}\delta_2 + \frac{3}{4}\delta_4$.
Figure \ref{fig:contMC} show the potentials $\mu \h U$, $\nu \h U$ and the resulting Root barrier.
\end{example}
\begin{figure}[h!]
\centering
\subfigure[R\'eduite, when it first touches $\nu\h U$ at 2]{
\includegraphics[width= 0.48\textwidth]{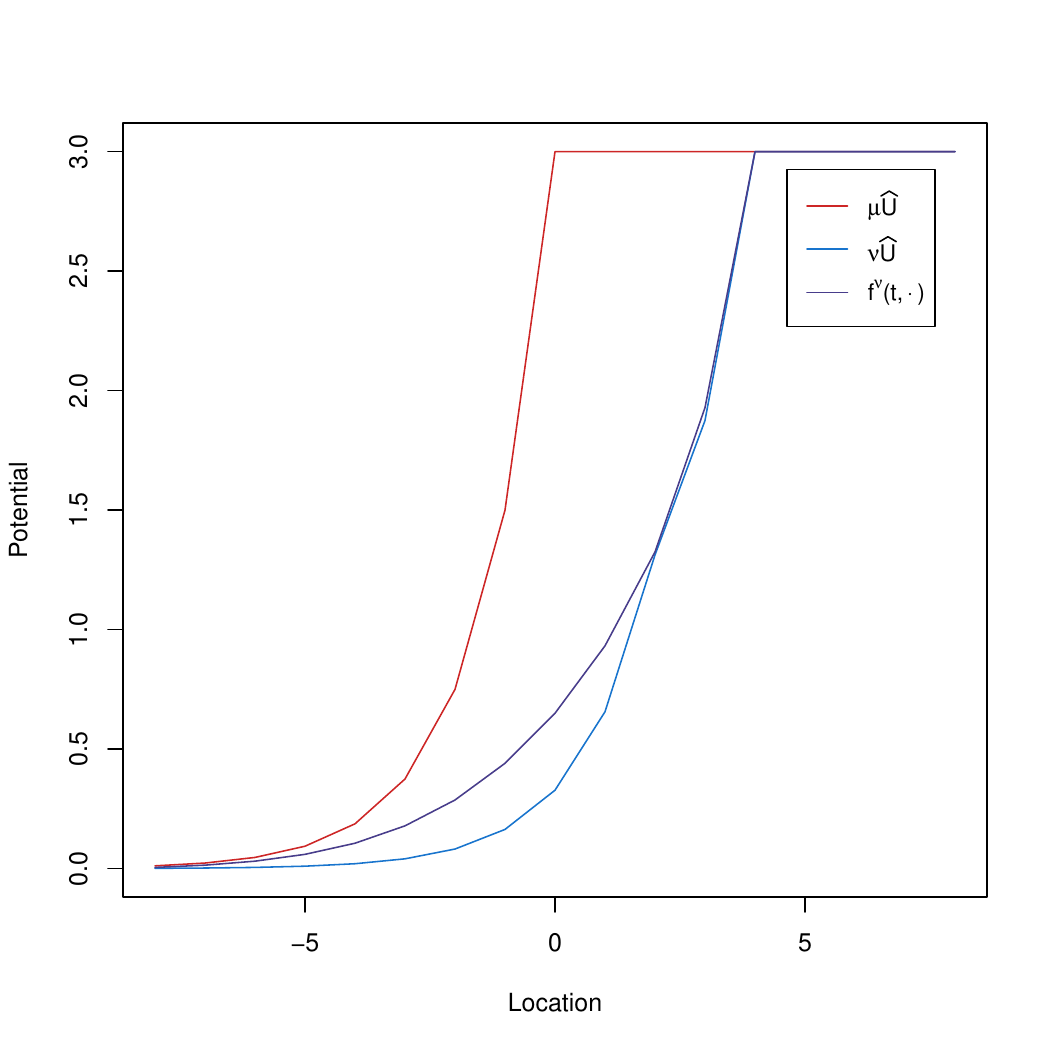}
}
\subfigure[Root barrier]{
\includegraphics[width= 0.48\textwidth]{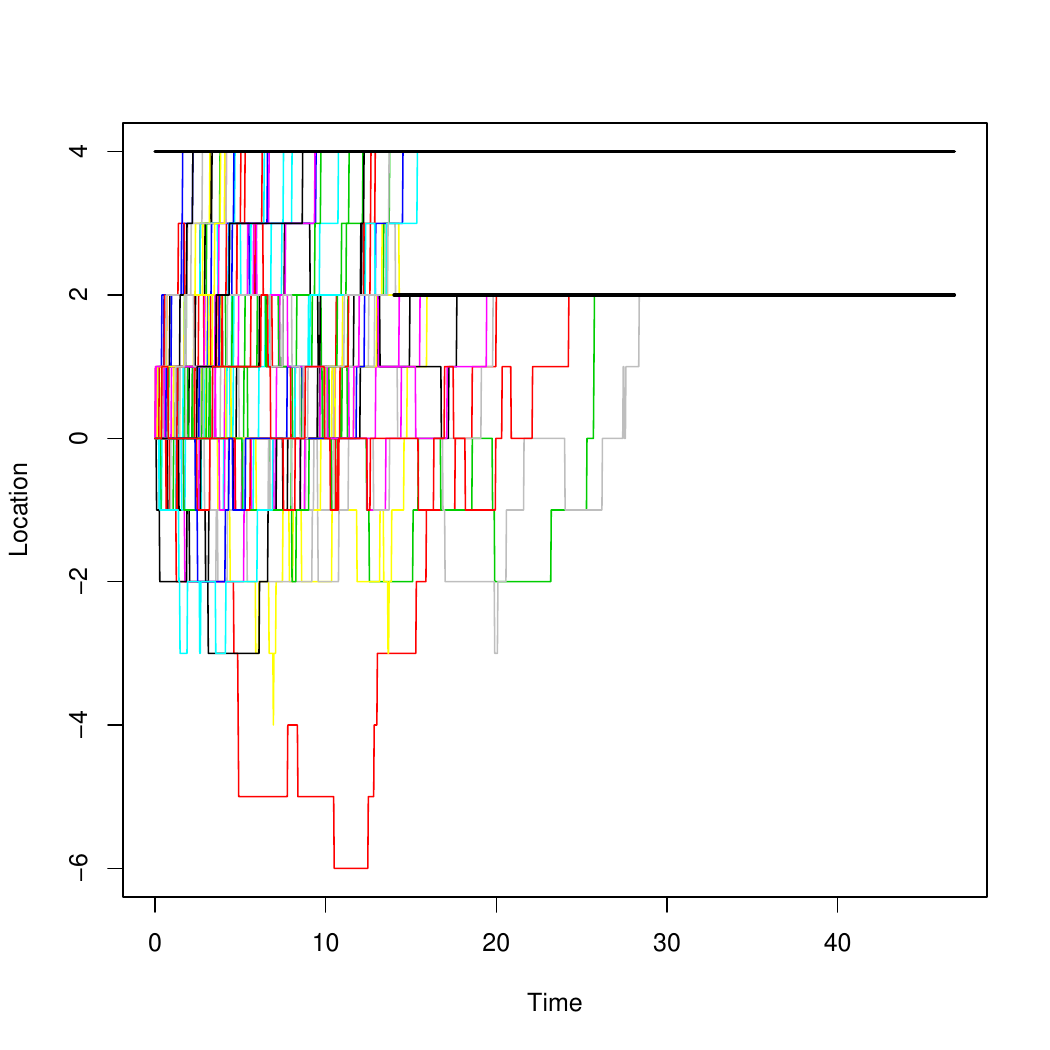}
}
\caption{Root embedding for the continuous-time asymmetric random walk on $\Z$ with $\lambda=1$, $p=\frac{2}{3}$, $\mu = \delta_0$ and $\nu = \frac{1}{4}\delta_2 + \frac{3}{4}\delta_4$.}\label{fig:contMC}
\end{figure}

\subsection{Hypo-elliptic diffusions}

Let $X$ be the diffusion in $\R^d$ obtained by solving an SDE formulated in the Stratonovich sense
\begin{equation} \label{eq:sdeX}
\dd X_t = \sum_{i=1}^N V_i(X_t) \circ \dd B_t + V_0(X_t) \dd t
\end{equation}
where the $V_i$, $i=1,\ldots, N$, are vector fields on $\R^d$ which we assume to be smooth with all derivatives bounded, and $B$ is a standard Brownian motion in $\R^N$.
We further assume that $X$ is  killed at rate $c(X) \dd t$, where $c\geq 0$ is a nonnegative smooth function. $X$ is then a standard Markov process on $\R^d$, with generator $\mathcal{L}$ which acts on smooth functions via
$$\mathcal{L} f = -c f + \left(V_0 + \sum_i V_i^2 \right) f = -c f + \sum_{i} b_i \partial_{i}f + \sum_{ij} \partial_i \left( a_{ij}  \partial_{j} f \right),$$
where the $b_i$ and $a_{ij}$'s are smooth functions which can be written explicitely in terms of the $V_i$. The formal adjoint of $\mathcal{L}$ with respect to Lebesgue measure is then given by
$$\h{\mathcal{L}} f = (-{\rm{div}} (b) - c) f - \sum_{i} b_i \partial_{i}f + \sum_{ij} \partial_i \left( a_{ij}  \partial_{j} f \right)$$
and we can choose smooth vector fields $\h{V}_i$'s such that $\h{\mathcal{L}} = (-\mathrm{div} (b) - c) f + \left( \h{V}_0 + \sum_i \h{V}_i^2\right)$. Assuming that
\begin{equation} \label{eq:div}
{\rm{div}} (b) + c \geq 0 \mbox{ on } \R^d,
\end{equation}
 we can then identify $\h{\mathcal{L}}$ with the generator of the Markov process consisting of the Stratonovich SDE
 \begin{equation} \label{eq:sdeXtilde}
 \dd \h{X}_t = \sum_{i} \h{V}_i(X) \circ \dd B_t + \h{V}_0(X) \dd t
 \end{equation}
killed at rate $(\mathrm{div} (b) + c)(\h{X}_t) \dd t$.

 In addition, assume that the vector fields satisfy the weak H\"ormander conditions
\begin{equation} \label{eq:hor1}
 \forall x \in \R^d,~~ \mathrm{Lie}\Big[ V_i, \;  \left[V_0, V_i\right], \; i \geq 1 \Big](x) = \R^d,
 \end{equation}
 \begin{equation} \label{eq:hor2}
 \forall x \in \R^d,~~ \mathrm{Lie}\Big[\h{V}_i, \;  \left[\h{V}_0, \h{V}_i\right], \; i \geq 1\Big](x) = \R^d,
\end{equation}
then the classical H\"ormander result \cite{hormander1967hypoelliptic} yields that the semigroups $P_t$, $\tilde{P}_t$ associated to $X$, $\tilde{X}$ admit (smooth) densities with respect to Lebesgue measure.
Therefore, $(P_t)$ and $(\h{P}_t)$ are in duality with respect to Lebesgue measure, as seen by 
\begin{align*}
 \frac{\dd}{\dd s} \left\langle P_{t-s} f, \h{P}_s g \right\rangle = \left\langle - \mathcal{L} P_{t-s} f, \h{P}_s g \right\rangle + \left\langle P_{t-s} f, \h{\mathcal{L}} \h{P}_s g \right\rangle = 0, 
\end{align*}
which yields that $\left\langle P_{t} f, g \right\rangle = \left\langle  f, \h{P}_t g \right\rangle$, first for $f,g$ smooth with compact support and then for all $f,g \geq 0$ Borel measurable by an approximation argument. In conclusion, we have obtained the following.
 
 \begin{proposition}
 Assume that \eqref{eq:div} and \eqref{eq:hor1}-\eqref{eq:hor2} hold. Then the process $X$ given by solving the SDE \eqref{eq:sdeX} satisfies Assumption \ref{asn:dual}, with $\h{X}$ given by the solution to \eqref{eq:sdeXtilde} and $\xi$ given by the Lebesgue measure on $\R^d$.

 \end{proposition}

\begin{example}[Brownian motion in $\R^d$]For $d\leq 2$, as Brownian motion is recurrent, for any positive Borel function $f$ we have either $Uf \equiv \infty$ or $Uf\equiv 0$. Therefore we consider the Brownian motion killed when exiting the unit ball $B_1(0)$, i.e. $\zeta = \inf\{t>0:~||X_t|| >1\}$. For any probability measure $\mu$ with density $f$ supported on $B_1(0)$, the potential $\mu \h U = f\h U$ is the unique continuous solution of $\frac{1}{2}\Delta v = f$ on $B_1(0)$ vanishing on $\partial B_1(0)$, and is given explicitely as $ f\h U(x) = \E^x \left[\int_0^\zeta f(\h X_t) \dd t\right]= \int u(x, y) f( y)\dd y $, where\begin{equation}
u(x,y) = \begin{cases}
- |x-y|, & d=1,\\
\frac{1}{\pi}\log\frac{1}{||x-y||}, &d=2.\\
\end{cases}
\end{equation}
In dimensions $d\geq 3$, Brownian motion is transient, and the potential is the Newtonian potential on $\R^d$: 
\begin{equation}
u(x,y) =c_d\cdot \frac{1}{||x-y||^{d-2}},
\end{equation}
where $c_d = \sfrac{1}{2}\pi^{-\nicefrac{d}{2}}\Gamma \big(\sfrac{1}{2}(d-2)\big)$.
For $d=1$ the balayage order reduces to the convex order which is easy to verify. In higher dimensions it is in general non-trivial to find measures in balayage order. 

Now we consider the two-dimensional Brownian motion starting in 0.
As an example for a measure $\nu$ which is not rotational symmetric and can be embedded in the two-dimensional Brownian motion, we take $\nu$ as the measure with the following density 
as an approximation of the marginal of the diffusion $Y$ which is generated by the operator $\e^{-x_1-x_2}\L$, here $\nu(A)=\P^0(Y_{0.1}\in A)$,
$$ \frac{\nu(\ddi x) }{\dd x} = C\exp\left[-2.5\cdot \big(a (x_1-\tilde x) -b (x_2-\tilde x)-\tilde x)\big)^2    -6\cdot 
\big(b(x_1-\tilde x) + a(x_2-\tilde x) -\tilde x \big)^2     \right],$$
where $C$ denotes a normalising constant, $a=\cos(\sfrac{\pi}{4})$, $b= \sin(\sfrac{\pi}{4})$ and $\tilde x = 0.15$.
The (empirical) density is respresented in Figure \ref{fig:2DBM}(a) on page \pageref{fig:2DBM}. We take $\nu$ of this form since $Y$ can be obtained as a time change via an additive functional of $X$ which implies $\delta_0 \prec \nu$ (we will show this explicitly in Section \ref{sec:generalRoot}) and we see $\nu \h U$ in \ref{fig:2DBM}(b). 
\begin{figure}[p]
\centering
\subfigure[Empirical density of $\nu$]{
\includegraphics[width=0.45\textwidth]{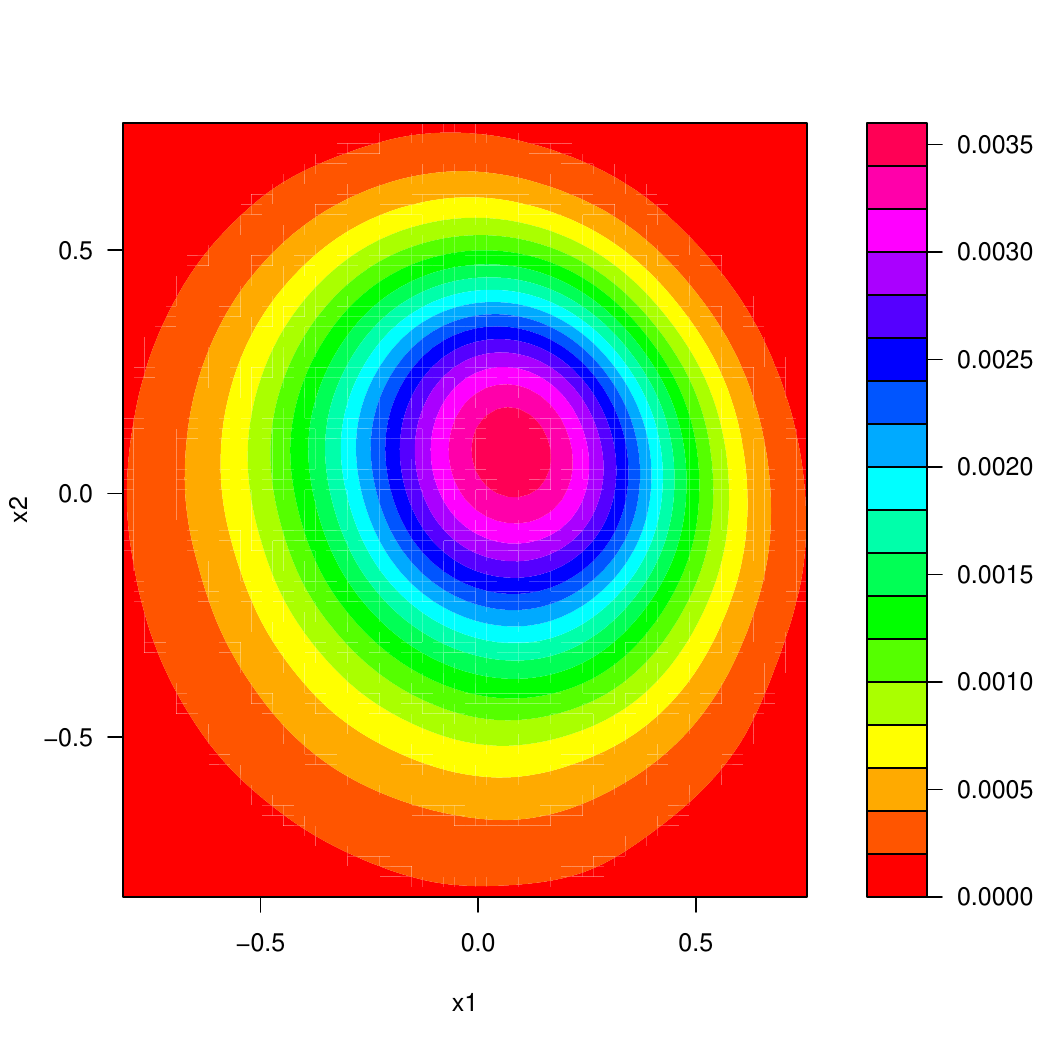} 
}
\subfigure[Potential difference $\delta_0\h U - \nu \h U$]{
\includegraphics[width=0.45\textwidth]{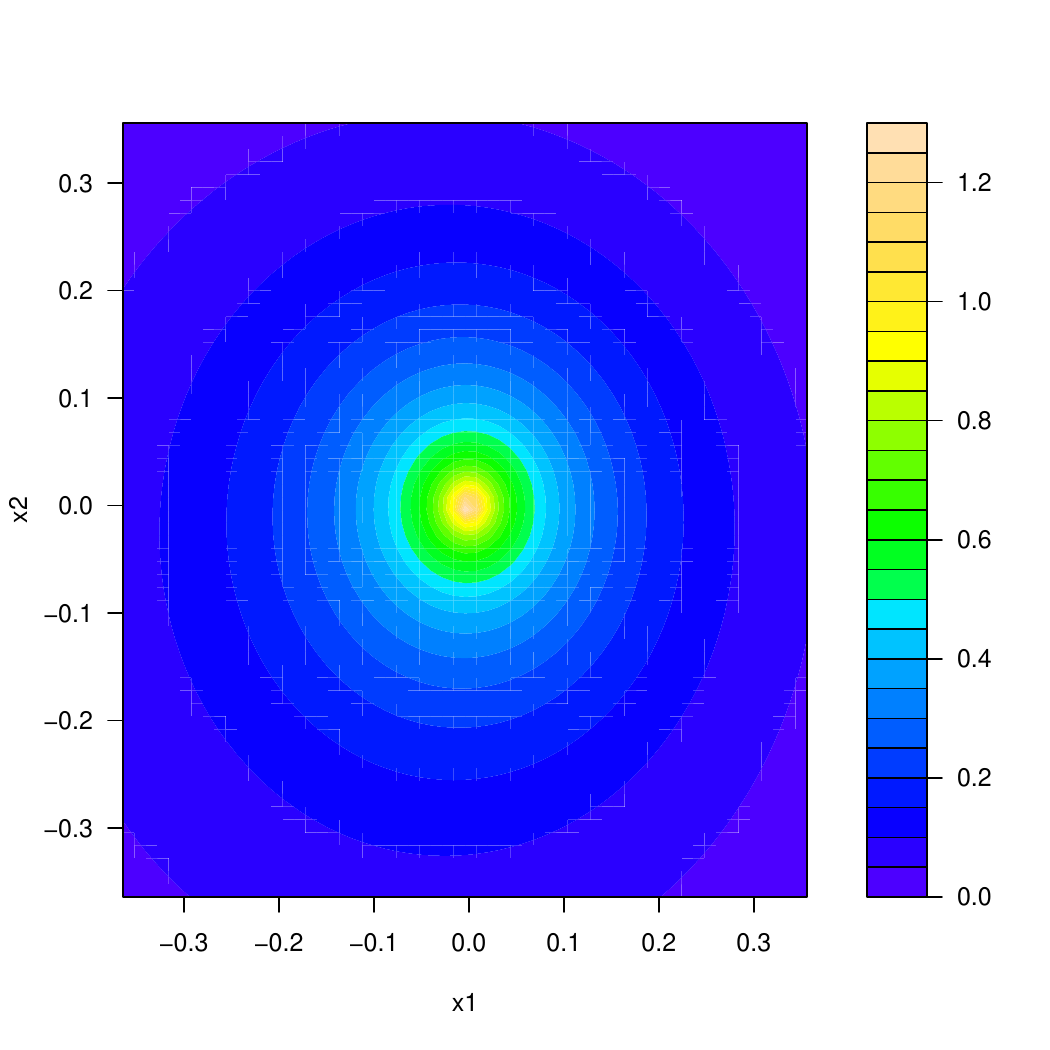} 
}
\subfigure[Brownian trajectories hitting the barrier]{
\includegraphics[width=0.45\textwidth]{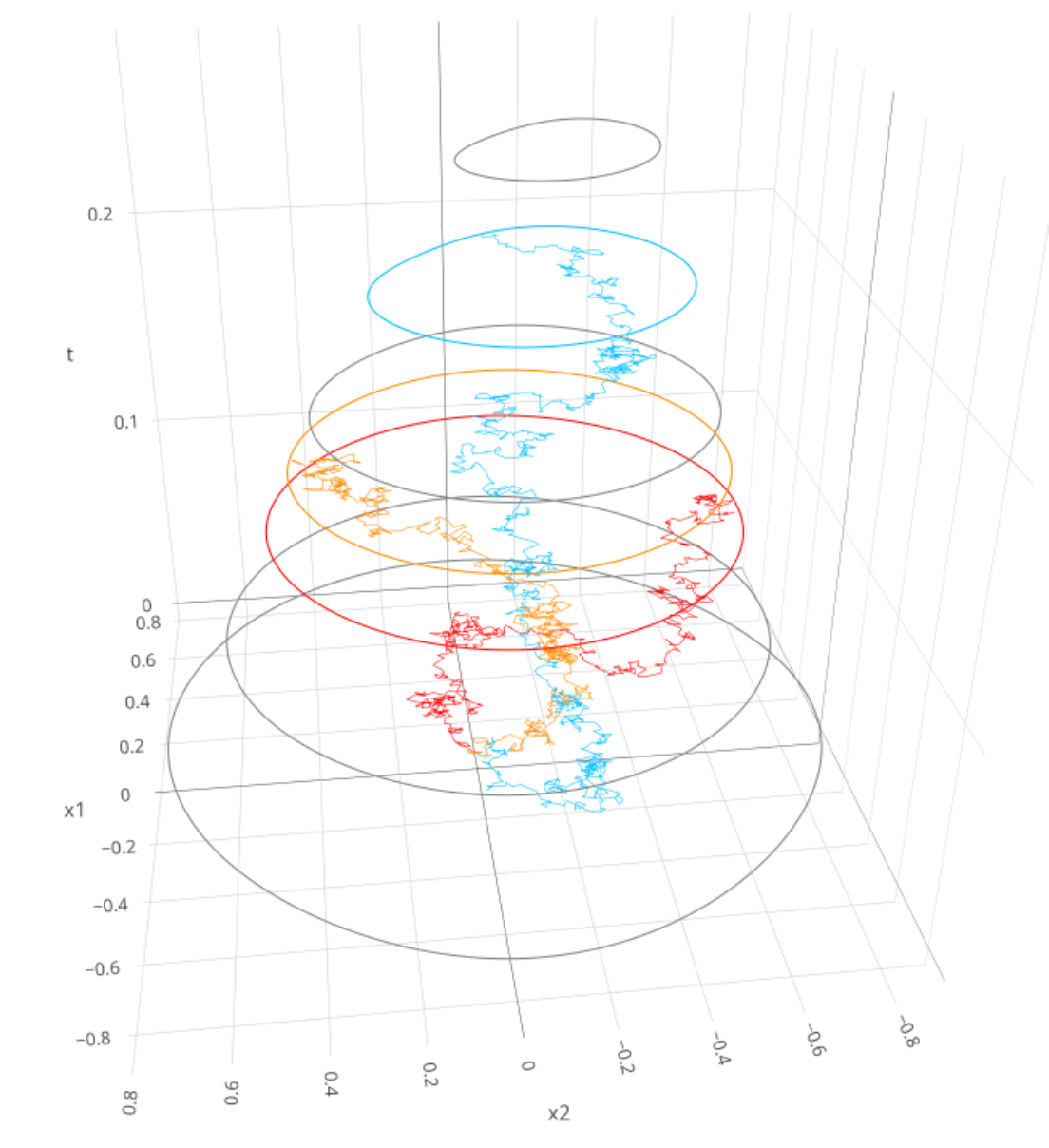} 
}
\subfigure[Root barrier]{
\includegraphics[width=0.45\textwidth]{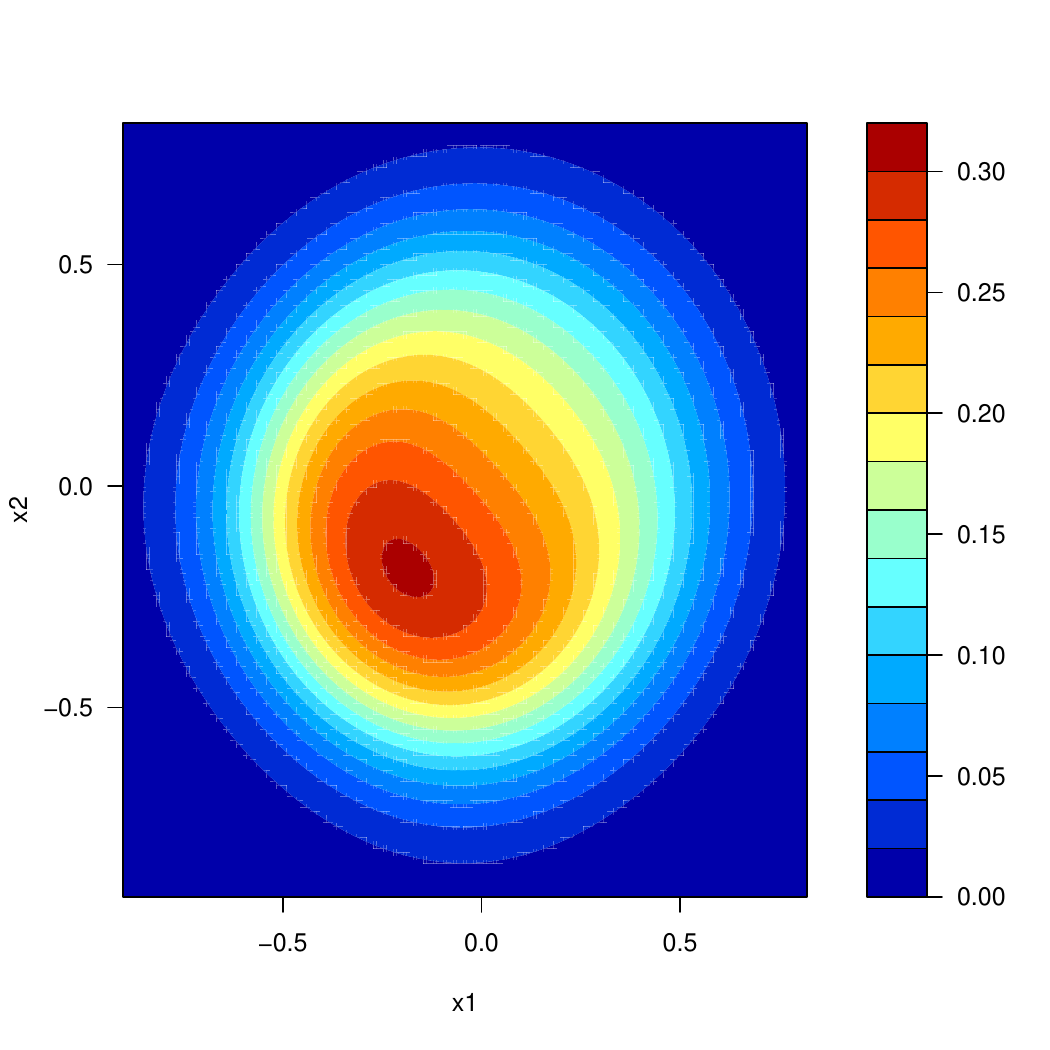} 
}
\caption{Root embedding for the 2d Brownian motion}\label{fig:2DBM}
\end{figure}
\end{example}

\begin{example}[Lie-group valued Brownian motion]
  Let $(B^1,B^2)$ be a two-dimensional Brownian motion.
  Then $(B^1,B^2,\int B^1 \dd B^2 - \int B^2 \dd B^1)$ can be identified (after taking the Lie algebra exponential), as a Brownian motion in the free nilpotent Lie group of order 2; see Appendix \ref{app:hypo-elliptic} for details and extension to general free nilpotent groups. 
The generator of the process is the sub-Laplacian $\Delta_G = \frac{1}{2} \left(X^2 + Y^2\right)$ on the Heisenberg group $G$; where in coordinates 
$$X= \partial_{x} + \frac{1}{2} y \partial_a, \;\;\; Y = \partial_{y} - \frac{1}{2} x \partial_a.$$
  As shown by \cite{gaveau1977, levy1951}, the transition density equals 
  \begin{align}\label{eq:densityGroupBM}
p_1(b^1,b^2,a) = \frac{1}{2\pi^2} \int_0^\infty \frac{x}{\sinh(x)}\exp\left(-\frac{(b^1+b^2)^2}{2 \tanh(x)}\right) \cos(ax)\dd x 
  \end{align}
and by Brownian scaling $p_t(b^1,b^2,a):=  p_1(\frac{b^1}{\sqrt t},\frac{b^2}{\sqrt t},\frac{a}{t})$.
In this case, it is already non-trivial to find measures $\mu,\nu$ in balayage order, $\mu \prec \nu$, even if $\mu$ is a Dirac at the origin.
However, Proposition \ref{Prop:Balayage Heisenberg} in the appendix shows that any measure $\tilde \nu$ on $(0,\infty)$ can be lifted to a measure $\nu$ on $G$ such that $ \delta_0  \prec \nu $. 
This provides a rich class of probability measures in balayage order, and Theorem~\ref{thm:main} allows to apply dynamic programming to compute the Root barrier solving SEP$(X,\delta_0,\nu)$.
However, this is computationally expensive since \eqref{eq:densityGroupBM} is not available in closed form. 
In this case, the well-posedness of the obstacle PDE 
\begin{align*}
\min \left[ (\partial_t -\Delta_G) u,~ u- \nu\h U\right] &=0,\\
u(0,\cdot) &= \delta_0\h U
\end{align*}
can be shown by standard methods (such as viscosity solutions). 
Again this leads to non-trivial numerics\footnote{We would like to thank Oleg Reichmann and Christian Bayer on helpful conversations and numerical experiments.}, even after using the radial symmetry of~\eqref{eq:densityGroupBM} to reduce the space dimension to 2, namely radius and area.  
Nevertheless, both approaches (dynamic programming and PDE) are applicable to compute barriers for group-valued Brownian motion, although much work remains to be done to turn this into a stable numerical tool and we leave this for future research. 
\end{example}

\subsection{Symmetric stable L\'evy processes}
A right-continuous stochastic process $(X_t)_{t\geq 0}$ is called an $\alpha$-stable L\'{e}vy process, if it has independent, stationary increments which are distributed according to an $\alpha$-stable distribution. We consider the symmetric case without drift.
In this case, the characteristic component is given by $\psi(\theta) = |\theta|^\alpha$, i.e. $\E[\e^{i\theta X_t}]= \e^{-t|\theta|^\alpha} =:g_t(\theta)$ and hence $X_t$ satisfies the scaling property $X_t \overset{d}{=} t^{\nicefrac{1}{\alpha}} X_1$. 
Classic results, e.g.~\cite{hartman1942infinitesimal}, show that $X$ has a transition density 
\begin{align*}
  p_t(x,y) = p_t(y-x) = (\F\ie  g_t)(y-x),
\end{align*}
which is absolutely continuous with respect to the Lebesgues measure. 
For further properties of symmetric stable processes, we refer to \cite{blumenthal1960some}.
We are going to take $\alpha\in (0,1)$, as in this case $X$ is transient, as shown in \cite{bogdan2009potential}. 
Furthermore, one-dimensional L\'evy processes $X$ are dual to $\h X:= -X$ with respect to the Lebesgue measure (see \cite{bertoin1998levy}). Since the jumps are distributed according to the symmetric stable distribution, the symmetric stable process $X$ is self-dual. By \cite{blumenthal2007markov} the potential $Uf(x) = \int u(x,y)f(y)\dd y$ equals
\begin{equation}
u(x,y) = C_{1,\alpha}\cdot |x-y|^{\alpha-1},
\end{equation}
where in the one-dimensional case $C_{1,\alpha}= \Gamma(\frac{1-\alpha}{2})\cdot\big[ 2^{\alpha}\sqrt{\pi}\Gamma(\frac{\alpha}{2})^{2}\big]\ie$. 
In order to construct the Root stopping time, we construct the function $f^{\mu,\nu}$ as described in Theorem \ref{thm:main} as solution to the obstacle problem
\begin{align*}
  v(0,\cdot) = \mu\h U, \;\;\; \min\left[(\partial_t + (-\Delta)^{\nicefrac{\alpha}{2}})v, v - \nu \h U \right] = 0,
\end{align*}
where the generator of the process $-(-\Delta)^{\nicefrac{\alpha}{2}}$ is given by the fractional Laplacian
\begin{equation}\label{eq:fracLap}
(-\Delta)^{\nicefrac{\alpha}{2}} f(y) = C_{2,\alpha} \cdot \operatorname{P.V.}\int_{-\infty}^\infty \frac{f(y)-f(z+y)}{|z|^{1+\alpha}}\dd z,
\end{equation}
with $\operatorname{P.V.}$~denoting a principal value integral. 
\begin{example}[{Embedding for $\alpha=0.5$, $\mu\sim \mathrm{Uniform}([-1,1])$ and $\nu\sim 0.75\cdot\mathrm{Beta}(2,2)$}]
Let $\mu$ be the Uniform distribution on $[-1,1]$, then 
\begin{equation}
\mu \h U (y) = \frac{C_{1,\alpha}}{2\alpha}\cdot \begin{cases}
(1-y)^\alpha + (1+y)^\alpha, & \text{ for } |y|<1,\\
(|y|+1)^\alpha - (|y|-1)^\alpha, & \text{ for } |y|\geq 1.
\end{cases}
\end{equation}
We want to construct a solution $T$ for \sep\, where the density of $\nu$ is given by $\frac{\nu(\ddi x)}{\dd x} = 0.75\cdot g_{a,b}$, where 
\begin{equation}
g_{a,b}(x) = \begin{cases}
\frac{\Gamma(a+b)}{\Gamma(a)\cdot \Gamma(b)}\cdot 2^{-a-b+1}\cdot (x+1)^{a-1}\cdot (1-x)^{b-1}, &|x|\leq 1,\\
0, &|x|>1.
\end{cases}
\end{equation}
is the density of a Beta$(a,b)$ distribution on the interval $[-1,1]$. Realisations of the resulting embedding will then give us that on the event $\{T<\infty \}$ where $\P(T<\infty) =0.75$, the stopped values $X_T$ are distributed according to the Beta$(a,b)$ distribution.
Studying general numerical methods for the fractional Laplacian is beyond the scope of this article, so we just discuss a quick method which is adapted to our case. 
We can rewrite \eqref{eq:fracLap} as
\begin{equation}
(-\Delta)^{\nicefrac{\alpha}{2}} f(y) = C_{2,\alpha} \cdot\lim_{h\rightarrow\infty}\int_0^h\frac{2f(y)-f(z+y)-f(y-z)}{|z|^{1+\alpha}}\dd z.
\end{equation}
Define the set $\mathcal{O}_{\overline T}:= [0, \overline T]\times [-K, K]$ for large $\overline T, K\in \R$ and $h:= (\Delta t, \Delta x) = \big(\frac{\overline T}{N_{\overline T}}, \frac{2K}{N_x}\big)$, where $N_x$, $N_{\overline T}\in \N$ are chosen large enough. The space-time mesh grid is defined as
\begin{align*}
  \mathcal{G}_h := \{t_n: t_n= n\cdot \Delta t,~ n = 0, 1, \dots, N_{\overline T}\}\times \{x_j: x_j = -K+j\cdot\Delta x,~ j = 1, \dots, N_x\}.
\end{align*}
For the resulting minimal excessive majorant of $\mu \h U \ind_{\{t\leq 0\}}+\nu \h U \ind_{\{t> 0\}}$ we expect that $f^{\mu,\nu}$ never touches $\nu\h U$ outside $[-1,1]$ as this is the support of $\nu$. Indeed, a straightforward calculation shows that starting in $\mu \h U$, for $|y|\gg 1$ we have $(-\Delta)^{\nicefrac{\alpha}{2}}\mu \h U(y) = o(|y-1|)$, i.e.~the repeated action of the fractional Laplacian on $\mu \h U$ outside an interval $[-K, K]$ with large enough $K\gg 1$ is negligible. For any $(t,x)\in \mathcal{G}_h$ we define the operator
\begin{equation*}
S^h[v^h](t,x) = \begin{cases}
v^h(t,x) +\Delta t\cdot \big((-\Delta)^{\nicefrac{\alpha}{2}}_h v^h(t, \cdot)\big)(x), & x\in[-K, K],\\
\mu \h U (x), & \text{else}.
\end{cases}
\end{equation*}
where $(-\Delta)^{\nicefrac{\alpha}{2}}_h$ is the evaluation of the fractional Laplacian using a Gau\ss-Kronrod quadrature as described in \cite{davis2007methods} on $\mathcal{G}_h$. Then the minimal excessive majorant $f^{\mu,\nu}$ for Theorem \ref{thm:main} can be computed on $\mathcal{G}_h$ as follows:
\begin{equation}
v^h(0, \cdot) = \mu \h U(\cdot), \qquad v^h((n+1)\Delta t, \cdot) = \max \big( \nu \h U(\cdot), ~S^h[v^h](n \Delta t, \cdot) \big).
\end{equation}
In Figure~\ref{fig:symstable} on page \pageref{fig:symstable}, we can see a realisation of the embedding for \sep\, with $\mu$ and $\nu$ as given above. 
As for small values of $\alpha$, the trajectories of $X$ may have large jumps, for the simulations we need to take into consideration that $X$ may jump back in the barrier although it already left the support of $\nu$. Following the results from \cite{blumenthal1961distribution, kyprianou2014hitting}, the probability of $X$ not returning to $(-1,1)$ after reaching level $x$ is
\begin{equation}
\P^x(X_t \not \in (-1,1) ~~\forall t>0)=\frac{\Gamma(1-\sfrac{\alpha}{2})}{\Gamma(\sfrac{\alpha}{2})\Gamma(1-\alpha)}\int_0^{\frac{x-1}{x+1}} u^{\nicefrac{\alpha}{2}-1}(1-u)^{-\alpha}\dd u.
\end{equation}
\end{example}
\begin{figure}[p]
\centering
\subfigure[Resulting barrier according Theorem \ref{thm:main} and the trajectories hitting the barrier among 100 samples]{
\includegraphics[width= 0.93\textwidth]{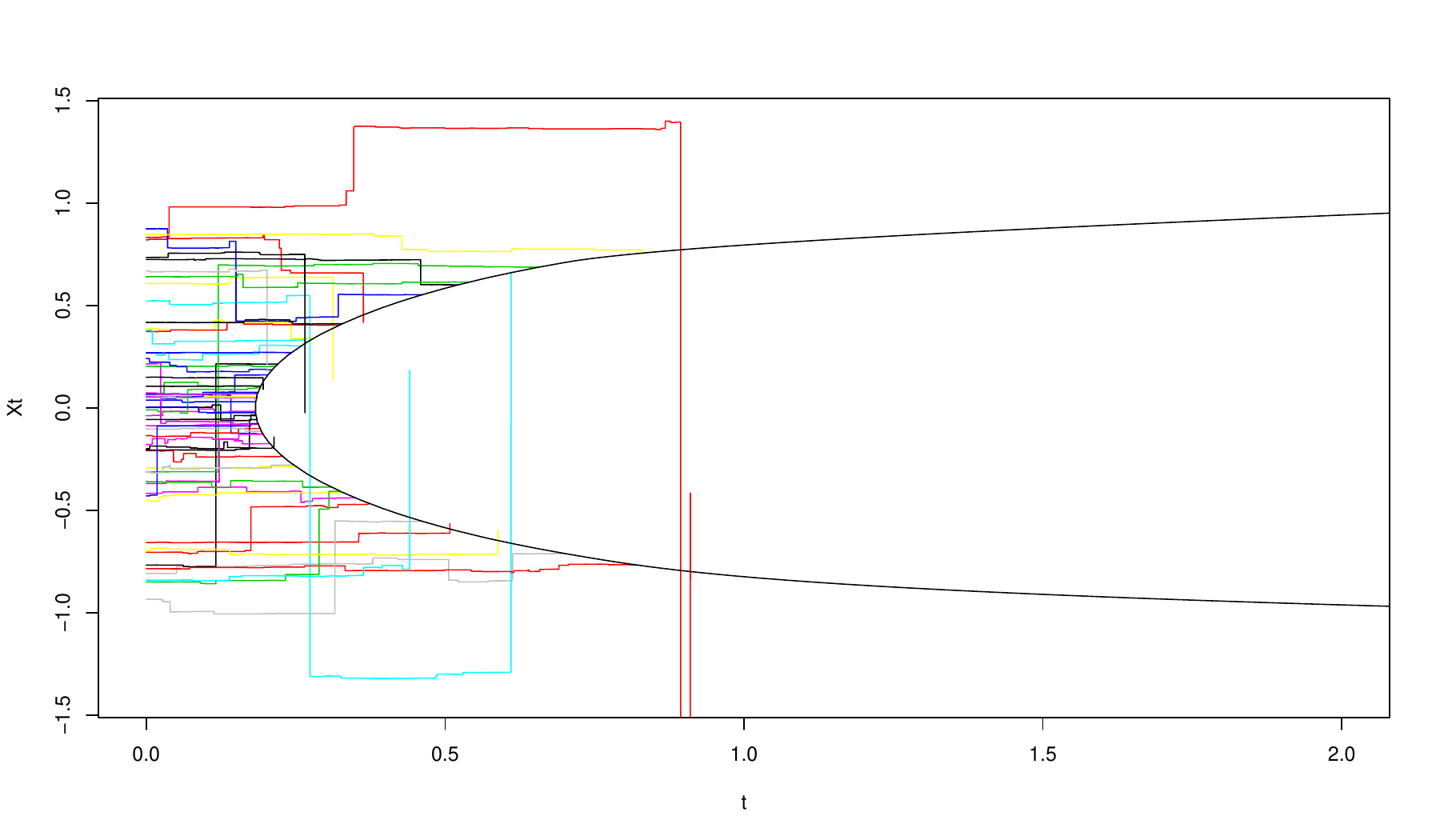}
}
\subfigure[Minimal excessive majorant, when it touches $\nu \h U$]{
\includegraphics[width= 0.45\textwidth]{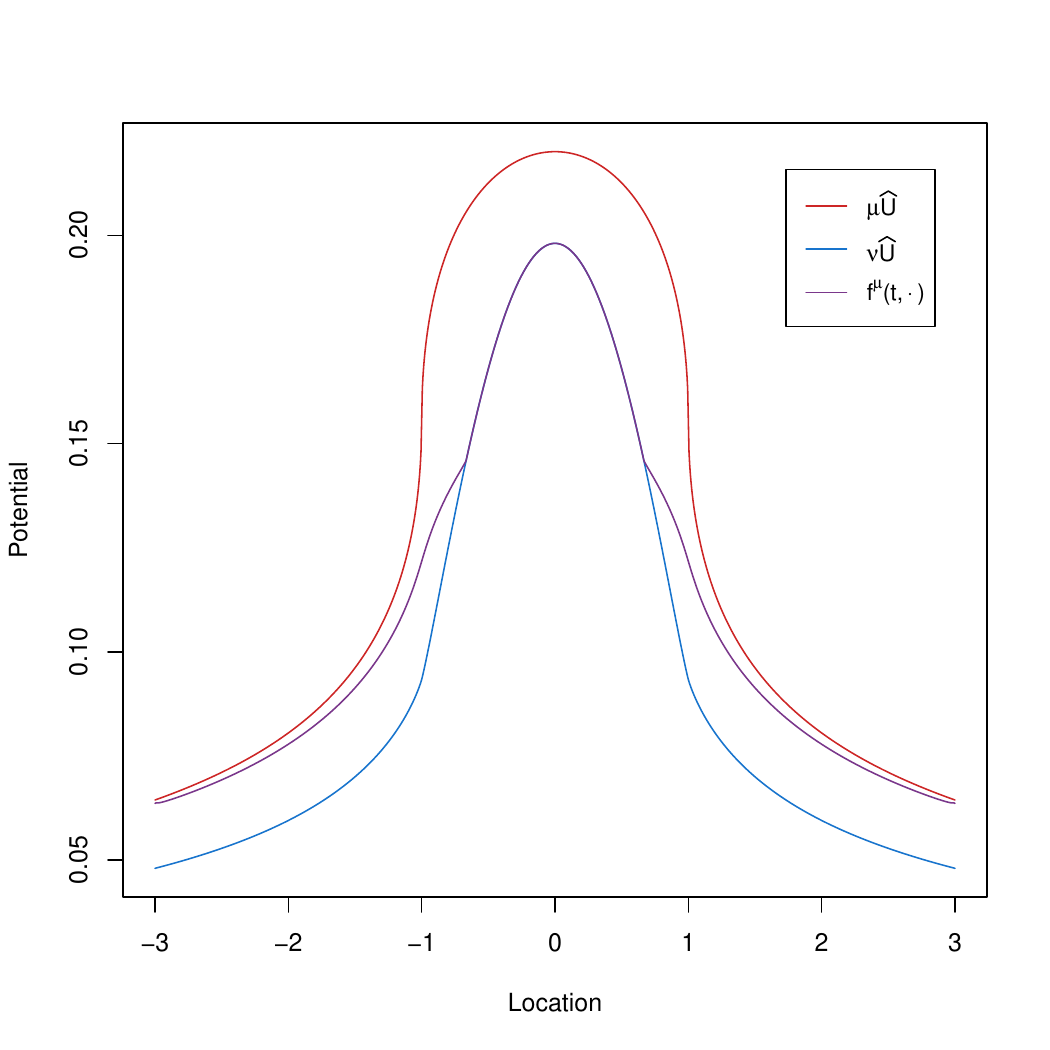}
}
\subfigure[Comparison of the quantiles of the Beta distribution and stopped values of 1000 trajectories]{
\includegraphics[width=0.45\textwidth]{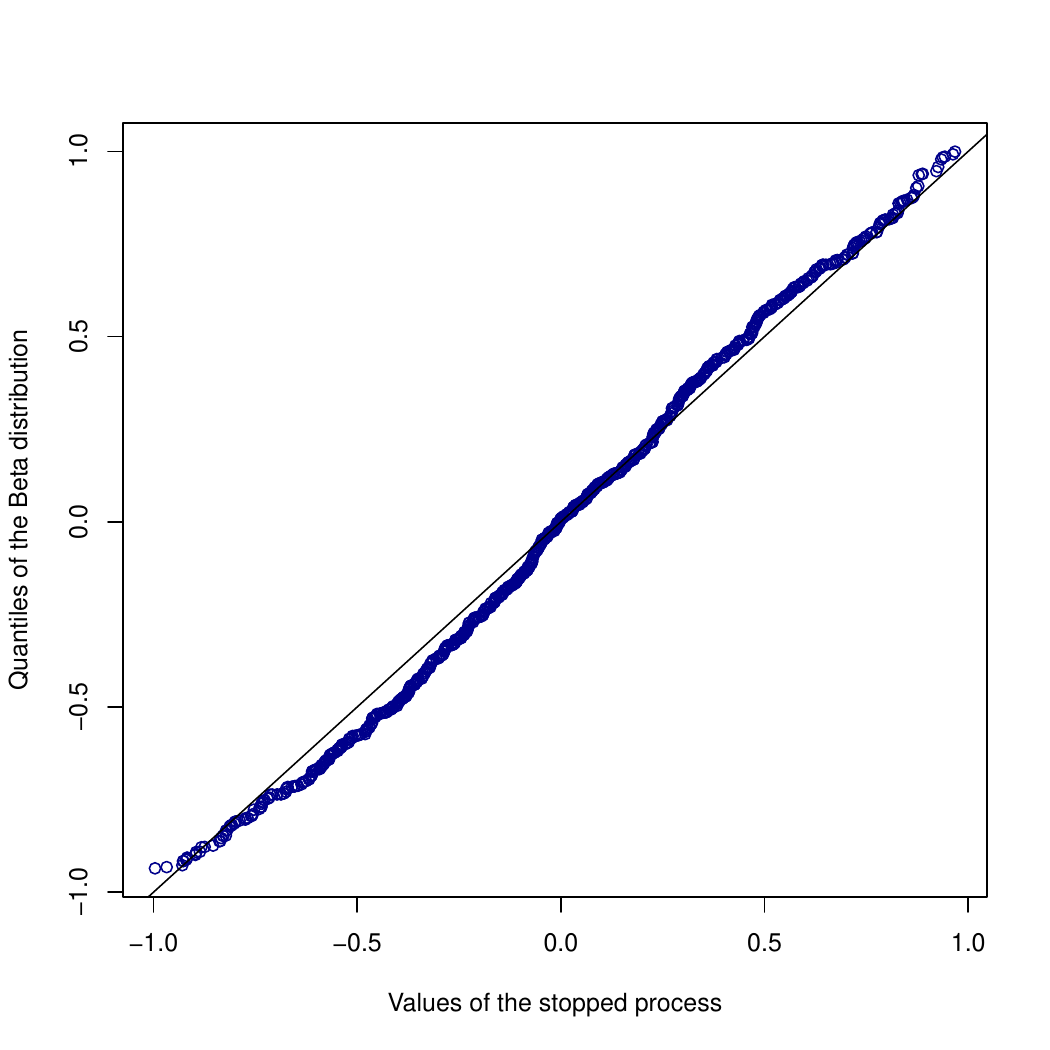}
}
\caption{Root embedding for the symmetric $\sfrac{1}{2}$-stable process, $\mu = \mathrm{Uniform}[-1,1]$, $\nu = 0.75\cdot \mathrm{Beta}(2,2)$.}\label{fig:symstable}
\end{figure}

\section{Towards generalised Root embeddings}
\label{sec:generalRoot}
The results of the previous sections, rely on Root's and Rost's approach to lift $X$ to a space-time process
\[
\overline X=(t,X_t)_{t \ge 0}
\]
and find a solutions of \sep\, that are given as a hitting time of $\overline X$. 
A natural generalisation is to replace the time-component by another real-valued, increasing process $A$ with $A_0=0$, such that $(A,X)$ is again Markov and carry out a similar construction.
That is, to construct a set such that its first hitting time by the lifted process
\[
(A_t,X_t) _{t \ge 0}
\]
solves \sep.
Again, one expects such a stopping time to be optimal in a minimal residual expectation sense, however, now formulated in terms of $A$.

Carrying out this program in full generality is beyond the scope of this article.
Instead, we focus on the case when $A$ is of the form $A_t=\int_0^t a(X_s)\dd s $ where $a$ is strictly positive.
Denote with $\tau_s:=\inf \{t>0: A_t = s\}$ the first hitting time of $s \ge 0$ by $A$ and with $Y_s:=X_{\tau_s}$ the time-changed process. 
Since for every (sufficiently nice) set $R \subset [0,\infty) \times E$  
\[
\inf \{ s > 0: (s,Y_s) \in R\} = \inf\{ s> 0: (A_{\tau_s},X_{\tau_s})  \in R\},
\]
this allows us to use the framework of the previous sections.
Concretely, one needs to verify that the assumptions of Theorem~\ref{thm:main} are met by $Y$.
This already provides a new class of solutions for \sep.
It can be seen as an interpolation between the Root embedding (when $a \equiv 1$) and the classical Vallois embedding \cite{vallois1992quelques}, since when applied to a Brownian motion, the classical Vallois embedding can be identified as the limiting case when $a$ approaches a Dirac at $0$.

\subsection{Generalised Root embeddings}
Below we restrict ourselves to additive functionals of the form 
\[\dd A_t = a(X_t) \dd t\]
with a Borel measurable $a$ which is locally bounded and locally bounded away from $0$, so that $t\mapsto A_t$ is one-to-one and the measure $m_A(\ddi x) = a(x) \xi(\ddi x)$ is \s-finite.
This implies that $A$ is an additive functional of $X$, i.e. $A$ satisfies
\begin{enumerate}
\item $A_0=0$, $t\mapsto A_t(\omega)$ is right continuous and non-decreasing, almost surely,
\item $A_t$ is $\F_t$-measurable,
\item $A_{t+s} = A_t+ A_s\circ \theta_t$ almost surely for each $t, s\geq 0$.
\end{enumerate}
\noindent We can then define the time-changed process $Y$ as follows
$$Y_t = X_{\tau_t}, \qquad \tau_t:= \inf\{ u > 0: \;\; A_u = t\}.$$
By \cite[Theorem 10.11]{dynkin1965markov}, $Y$ is a standard process. 
Its potential is given by 
\[U^Af(x) = \E^x\left[\int_0^ \infty f(X_t)\dd A_t\right] = \E^x\left[\int_0^ \infty f(X_t) a(X_t)\dd t\right] = \int u(x,y) f(y) a(y)\xi(\ddi y),\] 
and we can define the potential operator $\h U^Af(y) = \int f(x)u(x,y)  a(x)\xi(\ddi x)$ for any non-negative Borel-measurable function $f$ which corresponds to the time-changed process $\h Y_t = \h X_{\h \tau_t}$, where we analogously define $\h\tau_t := \inf\{ u > 0, \;\;\h A_u = t\}$ with $\h A_t = \int_0^t a(\h X_s) \dd s$.
In addition, strong duality holds,
\begin{theorem}[Revuz, Thm. V.5 and Thm. 2 in VII.3 in \cite{revuz1970mesures}]
The processes $Y$ and $\h Y$ are in strong duality with respect to the so-called Revuz measure $m_A(\ddi x) = a(x) \xi(\ddi x)$.
\end{theorem}
\begin{remark}\label{rem:tcpot}
From the duality with respect to the Revuz measure $m_A$, it follows for any Borel measure $\mu$ and $y\in E$ that
\[\mu \h U^A (y) = \int u(x,y)\mu(\ddi x).\]
Hence, $\mu \h U^A = \mu \h U$, i.e.~the potentials of the measures of the original and the time-changed process are equal.
However, note that $\mu U^A \not = \mu  U$.
\end{remark}
To apply our main result to the time-changed process we make the following assumption, which we will discuss later in this section. 

\begin{assumption} \label{asn:acA}
For all $t>0$ and $x\in E$, the transition functions of $Y$ and $\h Y$ are absolutely continuous with respect to $m_A$, i.e. $P^A_t(x,\cdot) \ll m_A$ and $\h{P}^A_t(\cdot,y) \ll m_A$.
\end{assumption}

Combining the above duality results with our main Theorem~\ref{thm:main} then gives us the following new solution of \sep. 
\begin{theorem}\label{thm:tc-embed}
Let $X$ be a Markov process and $A$ an additive functional for which Assumption \ref{asn:dual} and Assumption \ref{asn:acA} hold.
  Let $\mu\prec\nu$ be two measures with $\sigma$-finite potentials in balayage order, i.e.~$\mu \h U\geq \nu \h U$, and such that $\nu$ charges no semipolar set.
  Then there exists a Root barrier $R^A$ which embeds $(A,X)$ such that its first hitting time $T^A:= \inf\{ t >0: \;\; (A_t, X_t) \in R^A \}$ embeds $\mu$ into $\nu$,
\[
\mu P_{T} = \nu. 
\]
Moreover, if we denote 
\begin{equation}\label{eq:tc-red}
f^{A,\mu,\nu} = \inf \big\{ g\mbox{ $\h Q^A$-excessive:} \;\;\; g \geq \mu\h U(x) \ind_{\{t \leq 0\}} + \nu \h U(x) \ind_{\{t >0\}} \big\},
\end{equation}
where $\h Q^A$ denotes the space-time semigroup associated with $\h Y$, then 
\begin{enumerate}
\item $f^{A,\mu,\nu}(t,x) = \mu P_{\tau_{t}\wedge A_{T^A}}\h U(y)$,
\item 
$T^A = \argmin\limits_{S:~ \mu P_S = \nu}~~ \mu P_{\tau_t\wedge A_S} U^A(B)$ for all Borel sets $B$ and $t\geq 0$,  \label{thm:itm:addoptimal}
 \item  We may take
$R^A = \left\{ (s,x)\in \R_+ \times E\;\; |\;\; f^{A,\mu,\nu}(s,x) = \nu \h U (x) \right\}$.
\end{enumerate}
\end{theorem}

\begin{proof}
By Remark \ref{rem:tcpot},
$\mu \h U\geq \nu \h U$ implies $\mu \h U^A\geq \nu \h U^A$ for the time-changed process $Y$. 
We henceforth write $N_B(\omega) = \{t>0:~ X_t(\omega) \in B\}$ for the visits of a nearly Borel set $B$ during the lifetime of $X$.
Then the set $B$ is semipolar if and only if the set $N_B$ is almost surely countable.
Further we have \[\{s>0:~ X_{\tau_s(\omega)}(\omega) \in B\} \subseteq N_B(\omega)\] since the mapping $s\mapsto \tau_s$ is continuous and strictly increasing because $t\mapsto A_t$ is.
Therefore, any set $B$ which is semipolar for $X$ is also semipolar for $Y$ and $\nu$ does not charge sets which are semipolar for $Y$. 

Due to Assumption \ref{asn:acA}, the processes $Y$ and $\h Y$ and the measures $\mu$ and $\nu$ satisfy the assumptions of Theorem \ref{thm:main}.
Then $f^{A,\mu,\nu}$ and $R^A$ defined as above are exactly the equivalent results from Theorem \ref{thm:main} for $Y$ and the stopping time solving SEP$(Y,\mu, \nu)$ is given by
\begin{equation*}
T = \inf\{t >0:~ (t, Y_t) \in R^A\}.
\end{equation*}
Then for $f^{A, \mu, \nu}$ as in \eqref{eq:tc-red}, we have $f^{A, \mu, \nu}  (t, x) = \mu P^A_{t\wedge T}\h U^A (x)= \mu P_{\tau_{t\wedge T}}\h U(x)$ is the density of the measure $\mu P_{\tau_t\wedge A_S} U^A$ w.r.t. $m_A$.
If we define $T^A = \tau_T$, then for any nearly Borel set $B\in \mathcal{E}^n$ we obtain $\P^\mu (X_{T^A}\in B) = \P^\mu (Y_T\in B)= \nu (B)$ and it follows that for any solution $T$ to SEP$(Y,\mu, \nu)$, we have that $T^A= \tau_T$ is a solution for \sep. The optimality of Property 2 then naturally follows.

Finally, since $\inf\{s >0:~ (s, Y_s) \in R^A\} = \inf\{s >0:~ (A_{\tau_s}, X_{\tau_s}) \in R^A\}$, for $T^A= A_T$, we know that 
$$ T^A = \inf \{ t>0: ~ (A_t, X_t)\in R^A\},$$
which completes the proof.
\end{proof}

\begin{remark}
Assumption \ref{asn:acA} does not always hold, even under Assumption \ref{asn:dual}. For instance, let $X=(X^1,X^2)$ be the Markov process given by
$$\dd X_t = (\ddi B_t, a(X^1_t) \dd t)$$
where $a$ is non-negative, bounded, smooth with (for instance) $a'$ strictly positive, and $B$ is a linear Brownian motion. Then by the weak H\"ormander criterion, $X$ admits transition probabilities with respect to Lebesgue measure and satisfies Assumption \ref{asn:dual}. However taking the time-change $\tau_s$ corresponding to $\dd A_t = a(X_t^1) \dd t$, the resulting process satisfies
$$\dd Y_s = (\ddi B_{\tau_s}, \dd s)$$
which does not admit transition probabilities.
\end{remark}

\begin{remark}
Let $X$ be the diffusion with generator given in H\"ormander form by 
$$\mathcal{L}_X = V_0 + \sum_{i=1}^n V_i^2$$
(with for instance the $V_i$'s with bounded derivatives of all order), then (assuming $a$ also smooth,say) the generator of $Y$ is given by
$$\mathcal{L}_Y = \frac{1}{a} \mathcal{L}_X =V^A_0 + \sum_{i=1}^n \left(\frac{1}{a^{1/2}} V_i\right)^2$$
for some vector field $V^A_0$. In particular, if the following strong H\" ormander condition holds for $X$ :
\begin{equation*}
\forall x \in \R^d,~~ \mathrm{Lie}[V_1,\ldots,V_n](x) = \R^d
\end{equation*}
it also holds for the generator of $Y$, in which case $Y$ admits transition probabilities with respect to Lebesgue measure. This condition is for instance satisfied when $X$ is multi-dimensional Brownian motion (or more generally, Brownian motion on a Carnot group).
\end{remark}

\begin{remark}[Obstacle PDE]
The generator of the time-changed process $\h Y$ is given by $\L^A f(x) = \frac{1}{a(x)} \L f(x)$, see \cite{dynkin1965markov}.
Hence, we can again identify $f^{A,\mu,\nu}(t,x)$ as the solution to the obstacle problem
\[\min \left[ (\partial_t - a^{-1} \L) u, u - \nu\h{U} \right] = 0,\quad u(0,\cdot) = \mu \h{U} \qquad \text{ on } (0,+\infty)\times E\]
provided additional regularity assumptions are made that guarantee well-posedness of the above PDE. 
However, analogous to Corollary \ref{cor:OST}, dynamic programming applies without any additionally assumptions on $\L$ and $a$.  
\end{remark}
\begin{remark}[Vallois' embedding as limit of Root type embedding]
Item \eqref{thm:itm:addoptimal} in Theorem \ref{thm:tc-embed} implies that 
 $$T^A = \argmin \limits_{S:~ \mu P_S = \nu} \E^\mu [F(A_S)], \; \mbox{ for any  non-decreasing convex function $F$}.$$
Taking $X$ as the one-dimensional Brownian motion and $a(x) =\delta_0(x)$ the Dirac at $0$, the additive functional $A$ becomes the local time of $X$ at $0$.
  Thus -- at least informally since $a$ is not bounded from below -- Theorem \ref{thm:tc-embed} recovers the classical Vallois embedding, see e.g.  \cite{cox2007unifying, vallois1992quelques}. 
\end{remark}
\subsection{Examples} 
We now apply Theorem~\ref{thm:tc-embed} to concrete Markov processes. 

\begin{example}[Brownian motion $B$ and $A_t= \int _0^t \exp(2B_s)\dd s$]
Taking $X_t = B_t$ as the one-dimensional Brownian motion, the additive functional $\int _0^t \exp(2B_s)\dd s$ has received much attention (see e.g. \cite{matsumoto2005exponential}) due to application in mathematical finance in the context of Asian options. Then $B_{\tau_t} \de \log (Z_t)$, where $Z$ is the Bessel process of index 0 for which the transition density is well known (see \cite{matsumoto2005exponential}). 
Figure \ref{fig:BM-tc} on page \pageref{fig:BM-tc} shows the Root barriers for $\mu = \delta_0$ and $\nu=\mathrm{Uniform}[-1,1]$.
\end{example}
\begin{figure}[h!]
\centering
\subfigure{
\includegraphics[width= 0.48\textwidth]{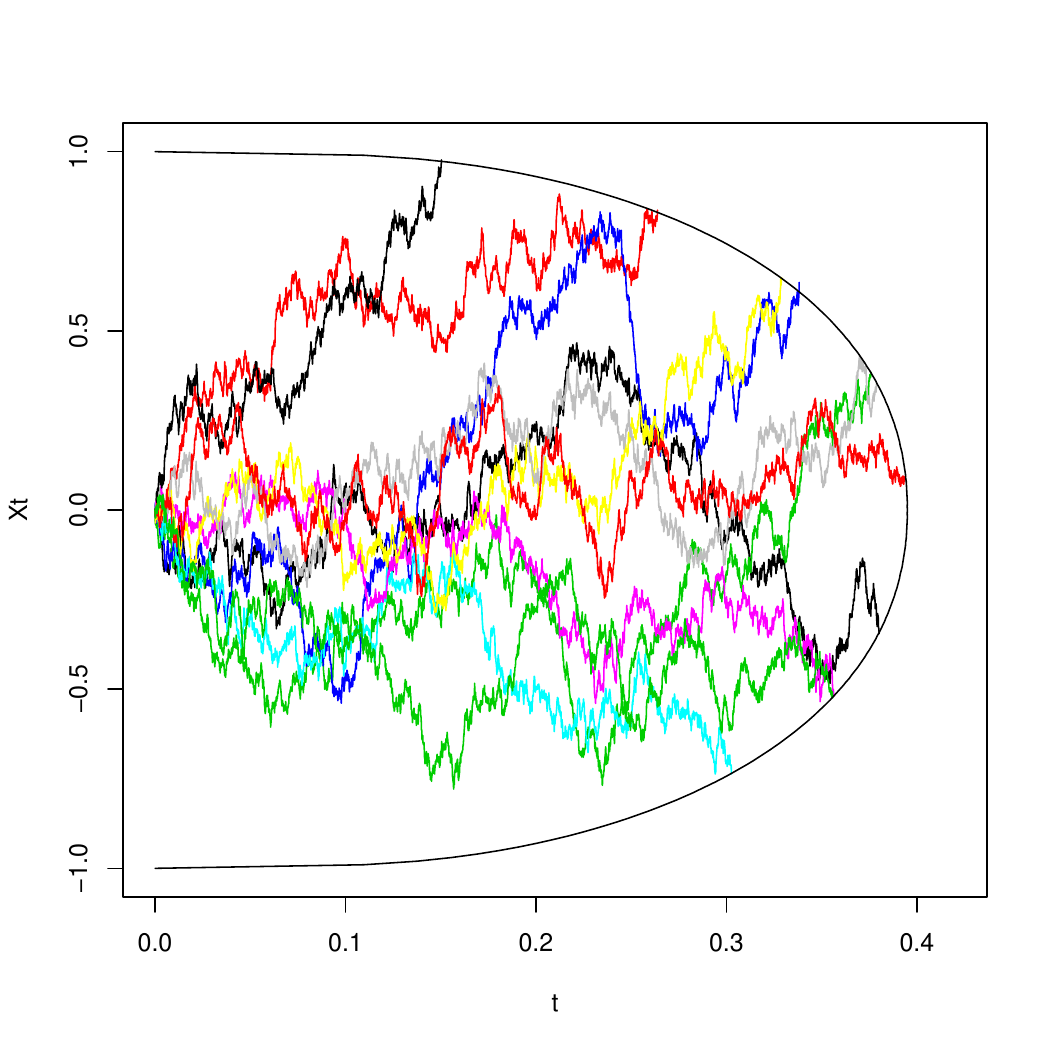}
}
\subfigure{
\includegraphics[width= 0.48\textwidth]{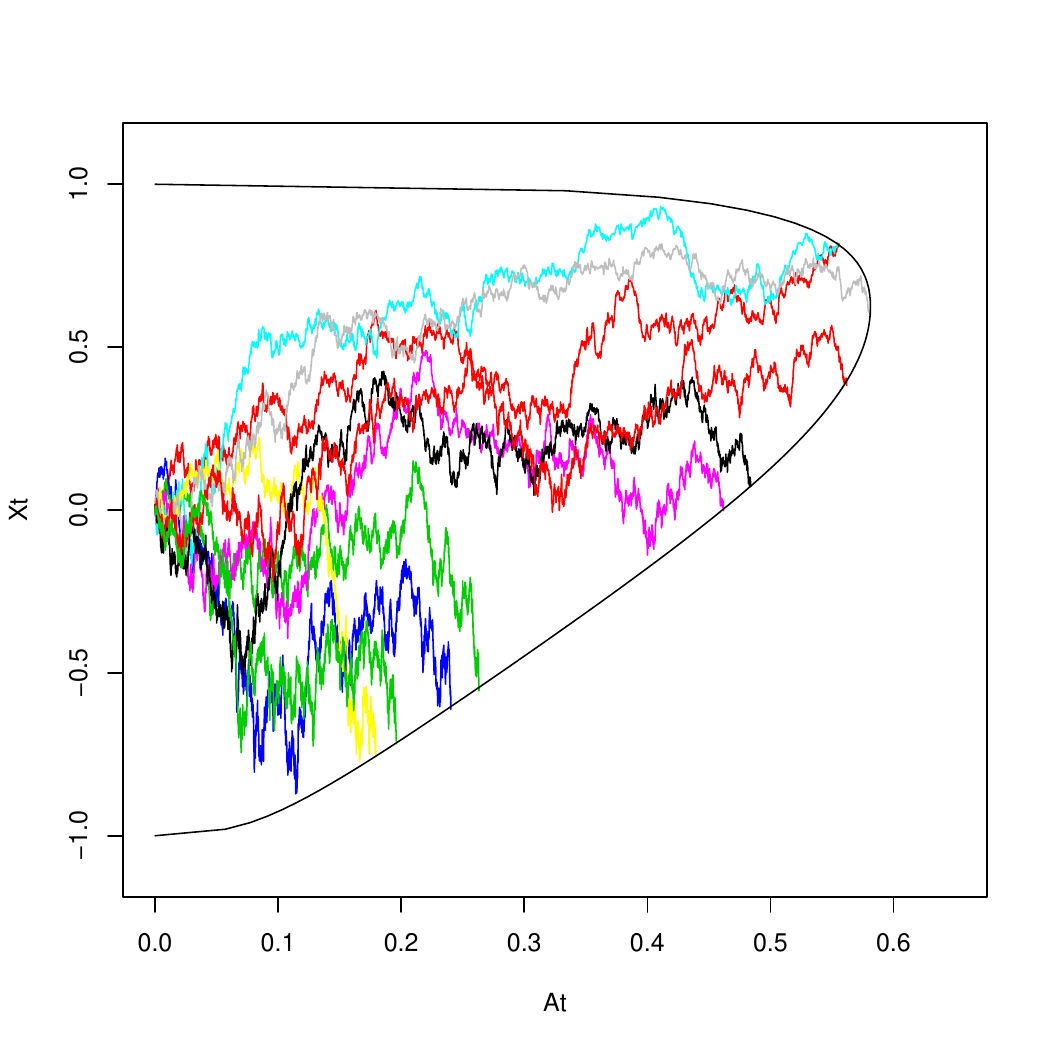}
}
\caption{Root embedding for Brownian motion and time-changed Brownian motion, with the same time discretisation in both pictures, $\mu = \delta_0$, $\nu=\mathrm{Uniform}[-1,1]$ }\label{fig:BM-tc}
\end{figure}
\begin{example}[Symmetric stable L\'evy process $X$ and $A_t=2t+ \int_0^t\mathrm{arctan}(4X_s)\dd s$]
For smooth $a$ with $c_1\leq a\leq c_2$ for some $c_1, c_2>0$, from \cite[Theorem (2.5)]{bichteler1983calcul}, the time-changed process $Y_t=X_{\tau_t}$ has absolutely continuous transition density with respect to the Lebesgue measure.
  Comparing the Root barriers for $\mu = \mathrm{Uniform}[-1,1]$ and $\nu = 0.75\cdot \mathrm{Beta}(2,2)$ for $X$ and $Y$, we can see that the barrier in Figure \ref{fig:tc:symstable}(b) on page \pageref{fig:tc:symstable} is not symmetric, unlike the barrier for $X$ in Figure \ref{fig:symstable}(b).
  Due to the time change, the process $Y$ runs faster past negative increments and more slowly through the positive parts which leads to hitting the barrier early on the negative parts and much later on the positive parts compared to $X$.
\begin{figure}[p]
\centering
\subfigure[Resulting barrier according Theorem \ref{thm:tc-embed} and the trajectories hitting the barrier among 100 samples]{
\includegraphics[width= 0.93\textwidth]{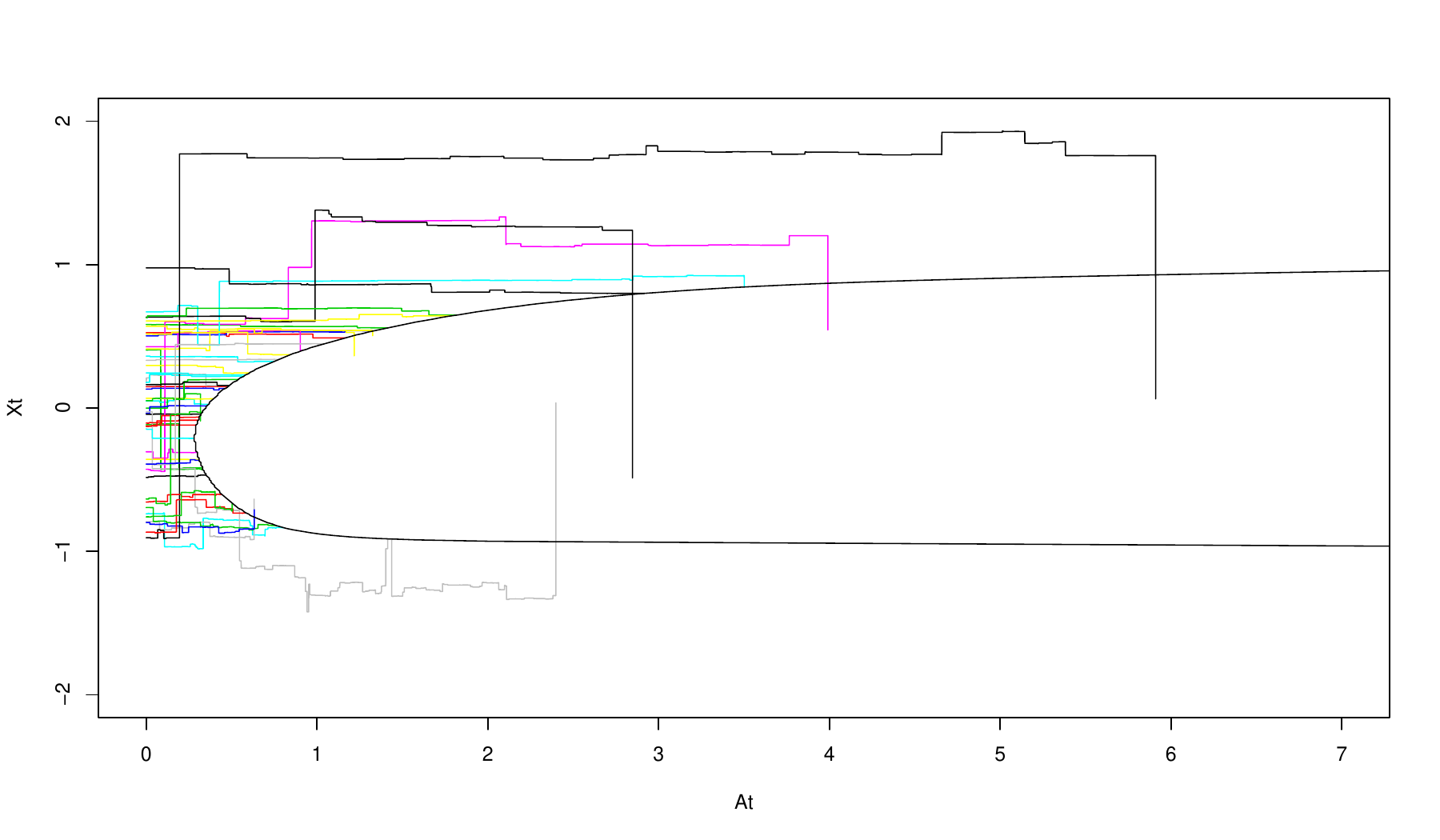}
}
\subfigure[Minimal excessive majorant, when it touches $\nu \h U$]{
\includegraphics[width= 0.45\textwidth]{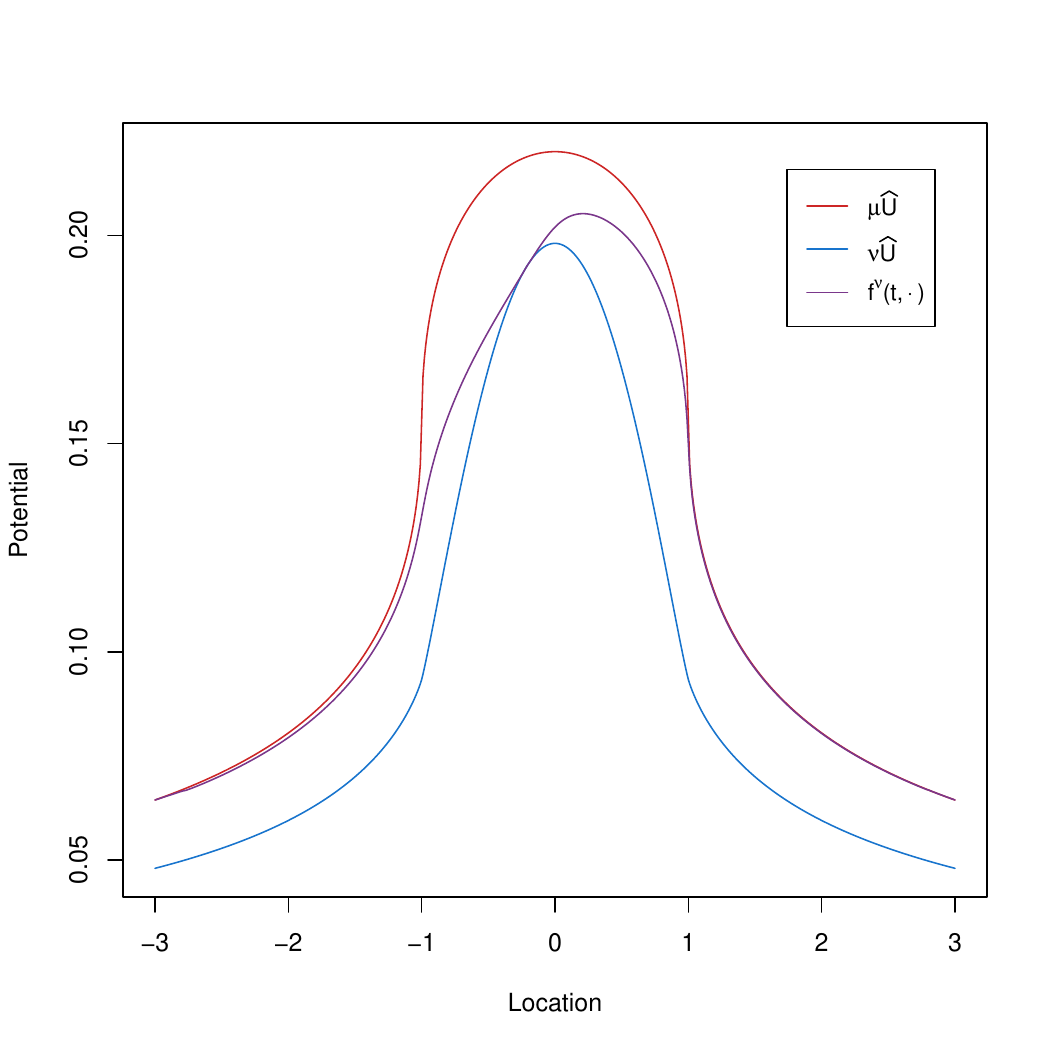}
}
\subfigure[Comparison of the quantiles of the Beta distribution and stopped values of 500 trajectories]{
\includegraphics[width=0.45\textwidth]{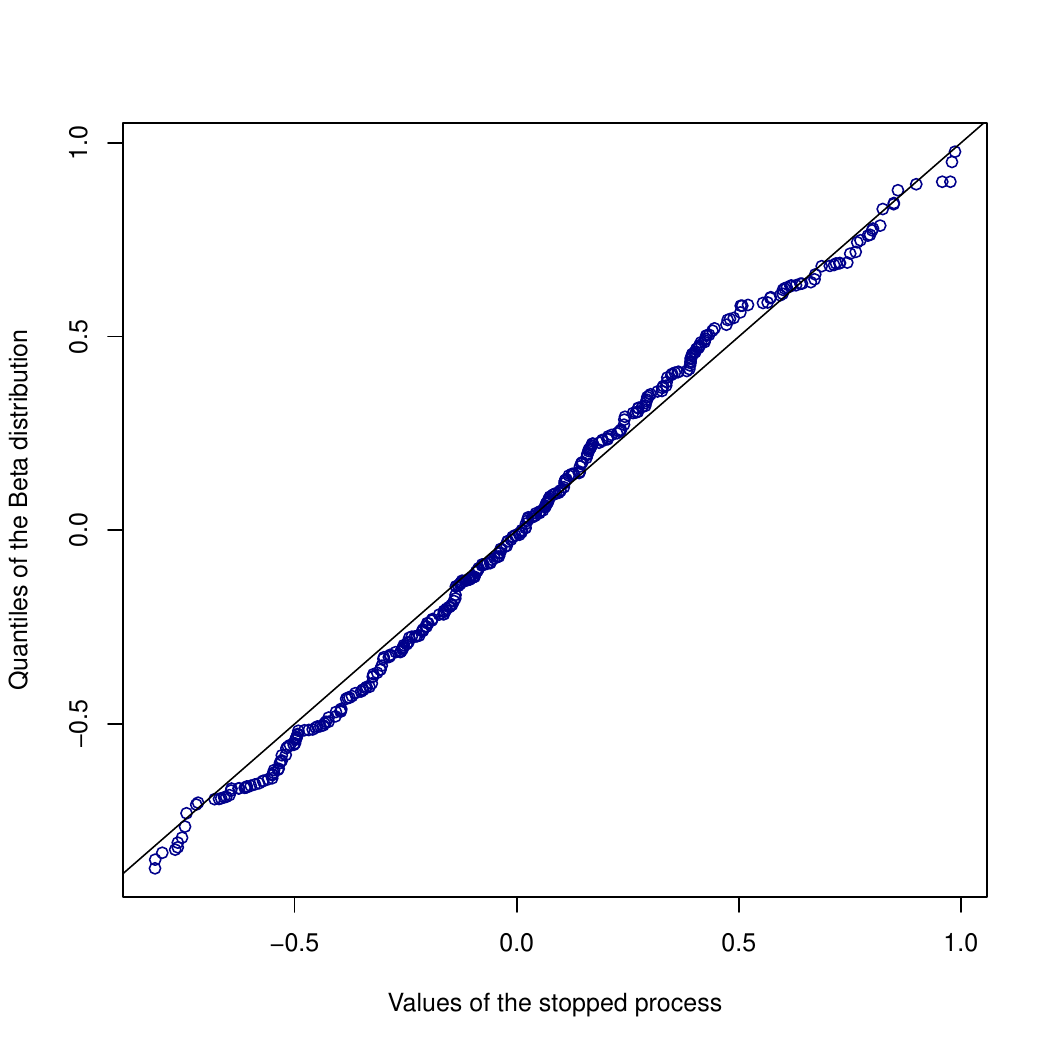}
}
\caption{Root embedding for the time-changed symmetric $\sfrac{1}{2}$-stable process, with $a(x) = 2+\mathrm{arctan}(4x)$, $\mu = \mathrm{Uniform}[-1,1]$, $\nu = 0.75\cdot \mathrm{Beta}(2,2)$.\label{fig:tc:symstable}}
\end{figure}
\end{example}

\newpage

\begin{appendices}
\section{Basic definitions from potential theory}\label{app:basic}
Below, we state some definitions taken from \cite[Chapter 0 and 1]{blumenthal2007markov}:

\begin{definition}[Universally measurable sets and nearly Borel sets]\label{def:univmeas}
Given a Borel \s-algebra $\mathcal{E}$, define the following \s-algebras on $E$:
\begin{enumerate}
\item the $\sigma$-algebra of \emph{universally measurable sets} $\mathcal{E}^\ast = \bigcap_{\mu \text{ finite}} \mathcal{E}^\mu$ given as intersection of completions $\mathcal{E}^\mu$ of $\mathcal{E}$ with respect to finite measures $\mu$,
\item  the $\sigma$-algebra  $\mathcal{E}^n$ of \emph{nearly Borel sets}.
  We call a set $B$ nearly Borel (with respect to $X$) if for each finite measure $\mu$ on $E$, there exists Borel sets $B_1 \subset B \subset B_2$ such that $\P^\mu(\exists t \geq 0:~X_t \in B_2 \setminus B_1) = 0$.
\end{enumerate} 
\end{definition}

\begin{definition}[Standard process]\label{def:stdproc}
 On a filtered probability space $\left(\Omega, \F, (\F_t)_{t \geq 0},  (\P^x)_{x \in E}\right)$, the stochastic process $X = (X_t)_{t\geq 0}$ with shift operator $\theta_t$ is called a
\emph{Markov process} with augmented state space $(E_\Delta, \mathcal{E}_\Delta)$, if for all $s,t\geq 0$,  $x\in E_\Delta$, and $B\in \mathcal{E}_\Delta$
\begin{enumerate}
\item $X_t$ is $\F_t$-$\mathcal{E}_\Delta$-measurable,
\item the map $x\mapsto \P^x(X_t\in B)$ from $E_\Delta$ to $[0,1]$ is $\mathcal{E}_\Delta$-measurable and $\P^\Delta(X_0 = \Delta)=1$,
\item $X_t\circ \theta_s = X_{t+s}$, and
\item $\P^x(X_{t+s}\in  B|\F_t) = \P^{X_t} (X_s\in B)$.
\end{enumerate} 
Furthermore, it is called a \emph{standard process}, if additionally
\begin{enumerate}
\item $(\F_t)_{t\geq 0}$ is right-continuous and $\F_t$ is complete with respect to the family of measures $\{\P^x, ~ x\in E\}$,
\item the sample paths $t\mapsto X_t(\omega)$ are c\`adl\`ag a.s.,
\item $X$ satisfies the strong Markov property, i.e. $X_T$ is $\F_T$-$\mathcal{E}^*_\Delta$-measurable and for all bounded measurable functions $f$ and $(\F_t)_{t\geq 0}$-stopping times $T$ we have $\E^x[f(X_{T+t})] = \E^x\big[\E^{X_T}[f(X_t)]\big]$ for all $x\in E$ and $t>0$, and
\item $X$ is quasi-left-continuous on $[0,\zeta)$, i.e.~for any increasing sequence $(T_n)_{n\in \N}$ of $(\F_t)_{t\geq 0}$-stopping times such that $T_n\uparrow T$ almost surely for a stopping time $T$, it holds that $X_{T_n}\rightarrow X_T$ almost surely on $\{T< \zeta\}$. 
\end{enumerate}
\end{definition}
\paragraph{Semigroup and potential.} In Table~\ref{tbl:semigroups} we let $x\in E$, $A\in \mathcal{E}^n$, $I\subseteq [0,\infty)$, $f:E\to\R$ be a $\mathcal{E}^\ast$-measurable function (extended to $E_\Delta$ by $f(\Delta)=0$), $\mu$ be a Borel measure on $E$, and $T$ be a stopping time.\\

\begin{table}[h!]
\centering
\begin{spacing}{1.2}
\begin{tabular}{L{0.35\textwidth}L{0.15\textwidth}L{0.5\textwidth}}
      \toprule
      \multicolumn{3}{c}{Markov process $X$} \\
      \midrule
semigroup $P=(P_t)_{t \ge 0}$ of $X$ & $P_tf$				& $P_t f(x)= \int P_t(x, \dd  y) f(y) = \E^x[f(X_t)]$ \\
& $\mu P_t$ 				& $\mu P_t (A)=\int \mu(\ddi x) P_t(x, A) = \P^\mu(X_t \in A)$ \\
potential $U =\int_0^\infty P_t \dd t$ of $X$  &$Uf$& $Uf(x)=\int U(x, \dd y) f(y)=\E^x\big[\int_0^\infty f(X_t) \dd t\big]$ \\
 & 
 $\mu U$  & $\mu U (A)=\int \mu(\ddi x) U(x,A) = \E^\mu\big[\int_0 ^\infty \ind_{\{X_t\in A\}}\dd t\big]$ \\
      \midrule
\multicolumn{3}{c}{Stopping times}\\
\midrule
semigroup at $T$ &$P_T$ & $P_T(x,\dd y)=\P^x\left(  X_T\in \dd  y; T <\zeta\right)$ \\
first hitting time&$T_A$& $T_A =\inf\{t>0:\,X_t\in A\}$, $\quad P_A= P_{T_A}$\\
       \bottomrule
\end{tabular}
\end{spacing}
\caption{Semigroups, potentials and stopping times for $X$ and its space-time lift $\overline X$.\label{tbl:semigroups}}
\end{table}

\noindent The potential $\mu U(A)$ of a measure $\mu$ on a set $A$ describes the occupation of the set $A$ by $X$ over its lifetime when starting in the initial distribution $\mu$; on the other hand, $Uf(x)$ evaluates the mass transported over the entire lifetime after starting in $x$ under $f$.
This explains the notation for the different actions $\mu U$ and $Uf$ of the potential kernel on $\mu$ and $f$ as we start in $\mu$ and end in $f$, respectively. 

\begin{definition}[Excessive functions and measures]
A non-negative $\mathcal{E}^\ast$-measurable function $f:E\to\R_+ \cup \{+\infty\}$ is called \emph{excessive} if $P_t f(x)\le f(x)$ for all $x\in E$ and $\lim_{t \downarrow 0}P_tf=f$ pointwise. 

Analogously, a Borel measure $\mu$ is called \emph{excessive} if it is $\sigma$-finite and $\mu P_t(A)\leq \mu(A)$ for all $A\in \mathcal{E}$ and $ t\ge 0$.
\end{definition}

\paragraph{Fine topology.}

In Table~\ref{tbl:fine}, the set $A$ denotes a nearly Borel set. 

\begin{table}[h!]
\begin{spacing}{1.2}
\begin{tabular}{L{0.28\textwidth}L{0.72\textwidth}}
      \toprule
      \multicolumn{2}{c}{Fine topology}\\
      \midrule
 polar set & $A$ is polar if $T_A=\infty$, $\P^x$-a.s.~for all $x\in E$\\
 thin set & $A$ is thin if $T_A>0$, $\P^x$-a.s.~for all $x\in E$\\ 
  semipolar set & $A$ is semipolar, if it is a countable union of thin sets\\
  $A^r$&the regular points of $A$; a point $x$ is called regular if
         $T_A=0$, $\P^x$-a.s.\\
fine topology on $E$ & topology where the open sets are $O\subseteq E$
                       s.t.~$T_{O^c}>0$, $\P^x$-a.s. $\forall x\in O$\\
 fine closure of $A$& the set $A\cup A^r$\\
        \bottomrule
\end{tabular}
\end{spacing}
\caption{\label{tbl:fine}The fine topology on $E$.}
\end{table}

\noindent Intuitively, the polar sets are those sets which are never visited at positive times by the process, while semipolar sets are those sets which are almost surely visited only countably many times by the process. Every polar set is semipolar, but the reverse implication is not true in general.

\section{Properties of the r\'eduite}\label{app:red}
\begin{definition}[R\'eduite]
Given a Markov semigroup $(P_t)$ associated to a standard process, and given $h \geq 0$ Borel-measurable and finely lower semicontinuous, we define the \emph{r\'eduite} (or smallest excessive majorant) of $h$ by
\begin{equation}
{\rm{Red}}_P(h) = \inf\left\{ f \mbox{ $P$-excessive}, \;\; f \geq h \right\}.
\end{equation}
\end{definition}

\begin{proposition}[Shiryaev, Lemma 3 and Theorem 1 in {\cite[Chapter 3]{shiryaev2007optimal}}]\label{prop:red1}
Let $X$ be a standard process with semigroup $(P_t)$ and $h\geq 0$ finely lower semicontinuous. Then:
\begin{enumerate}
\item ${\rm{Red}}_P(h) $ is excessive.
\item For all $x \in E$, it holds that 
$${\rm{Red}}_P(h)(x) = \sup_{\tau} \E^x\left[h(X_{\tau}) \ind_{\{\tau<\zeta\}}\right],$$
where the supremum ranges over stopping times $\tau$ taking values in $[0,\zeta)$.
\item Define for $\delta>0$ and $g \geq 0$ Borel-measurable, $R_\delta(g) = g \vee P_{\delta} g.$ Then it holds that
$${\rm{Red}}_P(h) = \lim_{n \to \infty} \lim_{N \to \infty} R_{2^{-n}}^N(h).$$
\end{enumerate}
\end{proposition}

Given a (positive) Borel measure $\gamma$, we similarly define
\begin{equation}
{\rm{Red}}_P(\gamma) = \inf\left\{ \lambda \mbox{ $P$-excessive measure}, \;\; \lambda \geq \gamma \right\}
\end{equation}
(note that the infimum above is the infimum of a family of measures, namely the smallest measure dominated by all measures in the family).

\begin{lemma} \label{lem:RedMtoF}
Assume that $X$ and $\h{X}$ are standard processes in strong duality with respect to a reference measure $\xi$. Let $h$ be finely lower semi-continuous and $\gamma(\ddi x) = h(x) \xi(\ddi x)$.
Then
$${\rm{Red}}_{{P}}(h) = \frac{\dd {\rm{Red}}_{\h{P}}(\gamma) }{\dd \xi}.$$
\end{lemma}

\begin{proof}
It is easy to see that ${\rm{Red}}_{\h{P}}(\gamma)$ is a $\h{P}$-excessive measure, and it therefore admits a ${P}$-excessive density $g$. Since ${{\rm{Red}}}_{\h{P}}(\gamma) \geq \gamma$, it holds that $g \geq h$ $\xi$-a.e. We then actually have the inequality everywhere since
$$ g = \lim_{t \to 0} P_t g \geq  \liminf_{t \to 0}P_t h \geq h $$
using the semicontinuity of $h$. Therefore $g \geq {\rm{Red}}_P(h)$.

For the opposite inequality, note that $\lambda(\ddi x):={\rm{Red}}_P(h)(x) \xi(\ddi x)$ is a $\h{P}$-excessive measure which dominates $\gamma$, so that ${\rm{Red}}_P(h) \geq g$ $\xi$-a.e., and then everywhere since both are excessive functions.
\end{proof}

\section{Hypo-elliptic Laplacian}\label{app:hypo-elliptic}
Denote with $(G_{n,d}, *)$ the free nilpotent group of depth $n$ of $d$ generators.
Denote with $V_1,\ldots,V_d$ a basis of $\R^d$ and identify it as left-invariant vector fields (the first level of the Lie algebra of $G_{n,d}$).  
If $B=(B^i)$ is a $d$-dimensional Brownian motion, then it is well-known that the solution of the SDE  
\begin{align*}
\dd X_t=\sum_{i=1}^d  V_i (X_t)\circ \dd B^i_t 
\end{align*}
is a (left)-Brownian motion\footnote{A process $X$ taking values in Lie-group $G$ is called (left) Brownian motion in $G$ if $t\mapsto X_t$ is continuous, $\left(  X^{-1}_sX_{t+s}\right)_{t \ge 0}$ is independent of
     $(X_u)_{0\le u \le s}$, and 
$\left(  X^{-1}_sX_{t+s}\right)_{t \ge 0}$ and $(X_t)_{t\ge 0}$ are identical in law. 
}
 on the Lie-group $G_{n,d}$.
Moreover, $X$ is a Markov process with generator $\frac{1}{2}\sum_{i=1}^d V_i^2$.
Following \cite{bonfiglioli2003subharmonic}, there exists a homogeneous norm $N: G_{n,d}\times G_{n,d} \to [0,\infty)$ such that 
 $$u(g*h\ie) =u(g, h) =c_q\cdot  N(g * h\ie)^{-q+2},$$
with a constant $q= q(n,d)$; moreover, $N$ is the fundamental solution to $-\Delta_G u(\cdot, h) = \delta_h$. In the case where $n=1$ and $d\geq 3$, $N$ is just given by the Euclidean norm on $\R^d$. For $n=d=2$, $G_{2,2}$ can be identified as the Heisenberg group.
It is more convenient to work in the associated Lie algebra $\mathfrak{g}_{2,2}=\log G_{2,2}$ which we identify in coordinates as $\R^3$.
Then $q=4$ and 
 \[N(g)=N(x,y,a) = \left(  (x^2+y^2)^2 + 16 a^2\right)^{\frac{1}{4}}.\]
For any measure $\mu$ on $G_{n,d}$ we then define 
$$ \mu \h U (h) = \int_{G_{n,d}} \mu (\ddi g)~ u(g, h).$$

\begin{lemma}\label{Prop:Balayage Heisenberg}
Let $\tilde{\nu}$ be any probability measure on $(0,\infty)$.
Define the measure
\begin{align}
\nu(\dd g)&:= \int_0^\infty \tilde{\nu}(\ddi r) \ind_{\partial D_N(r)}(g)\frac{|\nabla_{G}u|^2}{|\nabla u|}(g)~\sigma(\dd g),
\end{align}
where $D_N(r) = \{g\in G_{n,d}:~ N(g) < r\}$ is the open ball with respect to $N$, $\sigma$ is the surface measure and 
$$|\nabla_{G} N|^2 := \sum_{i=1}^n |V_i(N)|^2.$$
Then $\nu$ is a probability measures on $G_{n,d}$ such that $\delta_0 \prec \nu$ with respect to Brownian motion on $G_{n,d}$.
\end{lemma}
\begin{proof}
  From \cite[Theorems 3.4 and 4.1]{bonfiglioli2003subharmonic} we know that a upper semi-continous function $f$ is sub-harmonic on $G_{n,d}$ if and only if it satisfies the global-surface sub-mean property  for any $r>0$
\begin{equation}\label{eq:subharmonic}
f(g) \le \int_{\partial D_N(g,r)}\psi(h)~ f(h)~\sigma(\dd h) ,
\end{equation}
  where  $\psi(h):= \frac{|\nabla_{G}u|^2}{|\nabla u|}(h),$
  and  $D_N(g,r) = \{h\in G_{n,d}:~ N(g\ie h) < r\}$.
We write $u_0 (h) = u(0, h)$. Then $u_0 = \delta_0 \h U$ is subharmonic. 
The fact that $\nu$ is a probability measure follows as we have equality in \eqref{eq:subharmonic} for the harmonic function $f\equiv 1$.
For its potential we get
\begin{align*}
\nu\h U (g) &= \int_{G_{n.d}} u(g^{-1}*h) \nu(\ddi h) = \int_0^\infty  \left(\int_{\partial D_N(r)} \psi(h) u(g^{-1}h)\dd h\right)  \tilde \nu(\ddi r) \\
&= \int_0^\infty  \left( \int_{\partial D_N(g,r)} \psi(gh)u_0(h)\dd h\right)  \tilde \nu(\ddi r) \\
&\le \left(  \int_0^\infty \tilde \nu(\ddi r) \right)u_0(g\ie) = u_0(g)  = \delta_0\h U(g),
\end{align*}
where we used equation \eqref{eq:subharmonic} for $-u_0$ in the last inequality. 
\end{proof}

This in turns leads to an explicit density if we choose $\tilde \nu$ as the uniform measure on $(0,1)$.
\begin{corollary}
Let $n=d=2$. Then there exists a probability measure $\nu$ such that $\delta_0\prec \nu$ and that has a density with respect to Lebesgue measure given as 
\[
\nu(\ddi g) = \nu(\ddi(x,y,a)) =\frac{4}{\pi} \frac{x^2+y^2}{\sqrt{(x^2+y^2)^2+16 a^2}} \ind_{\{(x^2+y^2)^2+16a^2 <1\}}\dd x\dd y\dd a.
\]
\end{corollary}

\end{appendices}
\section*{Acknowledgements}
PG acknowledges the support of the ANR, via the ANR project ANR-16-CE40-0020-01. HO is grateful for support from the EPSRC grant “Datasig” [EP/S026347/1], the Oxford-Man Institute of Quantitative Finance, and the Alan Turing Institute.
CZ was supported by the EPSRC grant EP/N509711/1, via the project No. 1941799.

\bibliographystyle{plain}
{\small 

}
\end{document}